\documentclass[11pt]{amsart}
\usepackage{amsfonts,amsmath,amssymb}
\usepackage{graphicx}
\usepackage{epsfig}

\numberwithin{equation}{section}
\newcommand{\eps}{\varepsilon}
\newcommand{\diag}{\mathop{\mathrm{diag}}}

\newcommand{\E}{\mathsf{E}}
\newtheorem{lemma}{Lemma}[section]

\newtheorem{theorem}{Theorem}[section]

\newtheorem{remark}{Remark}[section]
\newcommand{\Pp}{\mathsf{P}}
\newcommand{\Pc}{\mathcal{P}}
\newcommand{\R}{\mathbb{R}}
\newcommand{\N}{\mathbb{N}}
\newcommand{\X}{\mathbb{X}}
\newcommand{\sgn}{\mathop{\rm sgn}}
\newcommand{\Law}{\mathop{\rm Law}}

\newcommand{\ONE}{{\bf 1}}
\newcommand{\Span}{\mathop{\rm span}}

\newcommand{\kap}{\varkappa}
\newcommand{\Cc}{\mathcal{C}}
\newcommand{\Wc}{\mathcal{W}}

\newcommand{\In}{\mathrm{In}}
\newcommand{\Out}{\mathrm{Out}}
\newcommand{\inin}{\mathrm{in}}
\newcommand{\out}{\mathrm{out}}
\newcommand{\z}{\mathbf{z}}

\author{
  Yuri Bakhtin
}
\address{School of Mathematics, Georgia Tech, Atlanta GA, 30332-0160, USA}
\email{bakhtin@math.gatech.edu}
\title{Noisy heteroclinic networks}

\begin{document}
\begin{abstract}We consider a white noise perturbation of dynamics in the neighborhood 
of a heteroclinic network. We show that
under the logarithmic time rescaling the diffusion converges in distribution in a special topology
to a piecewise constant process that jumps between saddle points along the heteroclinic orbits of the
network. We also obtain precise asymptotics for the exit measure for a domain containing the starting point
of the diffusion.
\end{abstract}
\maketitle

\section{Introduction}

In this note, we study small noise perturbations of a smooth continuous time dynamical system in the neighborhood of
its heteroclinic network.

The deterministic dynamics is defined
on $\R^d$ as the flow $(S^t)_{t\in\R}$ generated by a smooth vector field $b:\R^d\to\R^d$, i.e. $S^tx_0$ is the
solution of the initial-value problem
\begin{align}
\dot x(t)&=b(x(t)),\label{eq:ode}\\
x(0)&=x_0.\notag
\end{align}
We assume that the vector field $b$ generates a heteroclinic network, that is a set of isolated critical points connected
by heteroclinic orbits of the flow~$S$. 
Heteroclinic orbits arise naturally in systems with symmetries. Moreover, they are often robust under perturbations of the system
preserving the symmetries,
see the survey \cite{Krupa:MR1437986} and references therein, for numerous examples and 
a discussion of mechanisms of robustness.

We consider the system~\eqref{eq:ode} perturbed by uniformly elliptic noise:
\begin{align}
\label{eq:sde}
dX_\eps(t)&=b(X_\eps(t))dt +\eps  \sigma(X_\eps(t))dW(t),\\
X_\eps(0)&=x_0,\notag
\end{align}
where $W$ is a standard $d$-dimensional Wiener process, $\sigma(x)$ is a nondegenerate 
matrix of diffusion coefficients for every $x$, and $\eps>0$ is a small number. The initial point $x_0$ is assumed to belong
to one of the heteroclinic orbits.

Our principal result on the vanishing noise intensity asymptotics can be informally stated as follows.

{\bf Theorem.\ }{\it Under some technical nondegeneracy assumptions, as $\eps\to0$, the process $(X_\eps(t\ln(\eps^{-1})))_{t\ge 0}$ converges in distribution in an appropriate topology
to a process that spends all the time on the set of critical points and jumps instantaneously between them along the
heteroclinic trajectories.}

In fact, we shall provide much more detailed information on the limiting process and describe its distribution precisely. Thus, our
result provides a unified and mathematically rigorous background for the existing phenomenological studies, see e.g.~\cite{Armbruster--Stone--Kirk}. In particular, we shall see that in many cases the limiting process is not Markov. 
The precise description of the limiting process allows to obtain asymptotics of the exit distribution for domains
containing the starting point. This asymptotic result is of a different kind than the one provided by the classical Freidlin--Wentzell (FW)
theory, see~\cite{FW:MR1652127}. In fact, it allows to compute precisely the limiting probabilities of specific exit points 
that are indistinguishable from the point of view of the FW quasi-potential.

To prove the main result, we have to trace the evolution of the process along the heteroclinic orbits and in the neighborhood
of hyperbolic critical points. The latter was studied in \cite{Kifer} and \cite{Bakhtin-SPA:MR2411523}, where it was
demonstrated that in the vanishing diffusion limit, of all possible directions in the unstable manifold the system
chooses
to evolve along the invariant curve associated to the highest eigenvalue of the linearization of the system at the critical point, and the
asymptotics for the exit time was obtained. However, these results are not sufficient for the derivation of our main theorem, and more detailed analysis is required.

The paper is organized as follows. In Section~\ref{sec:example}, we study a simple example of a
heteroclinic network. In Section~\ref{sec:2d-linear-nonrigorous}, we give non-rigorous analysis of the behavior of
the diffusion in the neighborhood of one saddle point.
 The general
setting and the main result on weak convergence are given in Section~\ref{sec:setting}. To state the
 result we need to define in what sense the convergence is understood. Therefore, we begin our exposition in 
Section~\ref{sec:convergence_of_graphs}
 with a brief description of the relevant metric space, postponing  the proofs of all technical statements concerning
the metric space till Section~\ref{sec:basics-on-paths}. 
 In Section~\ref{sec:exit}, we state a result on the exit asymptotics and derive it from the main result
 of Section~\ref{sec:setting}. In Section~\ref{sec:main_proof}, we give the statement of the central technical lemma and
 use it to prove the main result. The proof of the central lemma is split into two parts. They are given in Section~\ref{sec:one_saddle} and 
 Section~\ref{sec:along_heteroclinic}, respectively. 
 Proofs of some auxiliary statements used in Section~\ref{sec:one_saddle}, are given in Section~\ref{sec:proof_tau_to_infty}.
 Section~\ref{sec:discussion} is devoted to an informal discussion of implications of our results and their extensions.

{\bf Acknowledgments.}  I started to think about heteroclinic networks in the Fall of 2006, after I attended a talk on their 
applications to neural dynamics by Valentin Afraimovich (see his paper~\cite{rabinovich:188103}). I am grateful to him for several stimulating discussions this work began with,
as well as for his words of encouragement without which it could not have been finished. I am also grateful to NSF for
their support via a CAREER grant.

\section{A simple example}\label{sec:example}
Here we recall a simple example of a heteroclinic network described in~\cite{Krupa:MR1437986}.  If $a_1<0$, the
deterministic cubic system defined by the drift of the following stochastic system
\begin{align}
dX_{\eps,1}&=X_{1,\eps}(1+a_1X_{1,\eps}^2+a_2X_{\eps,2}^2+a_3X_{\eps,3}^2)dt+\eps dW_1,\notag \\
dX_{\eps,2}&=X_{2,\eps}(1+a_1X_{\eps,2}^2+a_2X_{\eps,3}^2+a_3X_{\eps,1}^2)dt+\eps dW_2,\label{eq:example_system}\\
dX_{\eps,3}&=X_{3,\eps}(1+a_1X_{\eps,3}^2+a_2X_{\eps,1}^2+a_3X_{\eps,2}^2)dt+\eps dW_3 \notag
\end{align}
has 6 critical points 
\begin{align*}
z^{\pm}_1=(\pm\sqrt{-1/a_1},0,0),\\
z^{\pm}_2=(0,\pm\sqrt{-1/a_1},0 ),\\
z^{\pm}_3=(0,0,\pm\sqrt{-1/a_1}).
\end{align*}
The matrix of the linearization of the system at $z^+_1$ is given
by
\[\diag\left(-2,\frac{a_1-a_3}{a_1},\frac{a_1-a_2}{a_1}\right),\] and the linearizations at all other critical
points are can be obtained from it using the symmetries of the system. We see that if
\begin{equation}
 \label{eq:saddle_conditions}
 a_3>a_1\quad\text{and}\quad a_2<a_1,
\end{equation}
then all the critical points are saddles with one unstable direction corresponding to the eigenvalue
$(a_1-a_3)/a_1$. It is shown in~\cite{Krupa:MR1437986} that system~\eqref{eq:example_system} admits 12
orbits connecting $z_1^\pm$ to $z_2^\pm$, $z_2^\pm$ to $z_3^\pm$, and $z_3^\pm$ to $z_1^\pm$. Each of these orbits lies
entirely in one of the coordinate planes.

Let us equip~\eqref{eq:example_system}  with an initial condition $x_0$ on one of the 12 heteroclinic connections, say,
on the one connecting $z_1^+$ to $z_2^+$ denoted by $z_1^+\to z_2^+$.  

The theory developed in this paper allows to describe precisely the limiting behavior of the process as $\eps\to0$.
Namely, the process $X_\eps$ will stay close to the heteroclinic network, moving mostly
along the heteroclinic connections between the saddle points. At each saddle point it spends a time of order  
$\ln(\eps^{-1})$. More precisely,
the process $Z_\eps$ defined by $Z_\eps(t)=X_\eps(t\ln(\eps^{-1}))$ converges to a process
that jumps from $x_0$ to $z_2^+$ instantaneously along the heteroclinic connection $z_1^+\to z_2^+$, sits at $z_2^+$
for some time, then chooses one of the orbits $z_2^+\to z_3^+$ and $z_2^+\to z_3^-$, and jumps {\it along that
orbit}
instantaneously, spends some time in the endpoint of that orbit and then chooses a new outgoing orbit to follow,
etc.
However, the details of the process depend crucially on the eigenvalues of the linearization at the critical point.

 If at each saddle point the contraction is stronger than expansion, i.e.,\ if
\[
 \frac{a_1-a_2}{a_1}<-\frac{a_1-a_3}{a_1},
\]
which is equivalent to $a_2+a_3<2a_1$, then the system exhibits loss of memory and the sequence of saddle points
visited by the limiting process is a standard random walk on the directed graph formed by the network, i.e.\ a Markov
chain
on saddle points that at each point chooses to jump along one of the two possible outgoing connections with probability $1/2$.
 
If the expansion is stronger than the contraction, i.e., $a_2+a_3>2a_1$, then at $z_2^+$, the first saddle point
visited, the process still chooses each of the two possible next heteroclinic connections with probability $1/2$.
However, if it chooses $z_2^+\to z_3^+$, then it will cycle through $z_1^+,z_2^+,z_3^+$ and never visit any other
critical points. If it chooses $z_2^+\to z_3^-$, then it will cycle through $z_1^+,z_2^+,z_3^-$ and never visit any
other critical points.
This situation is strongly non-Markovian, because the choice the system makes at $z_2^+$ at any time is determined by
its choice during the first visit to that saddle point.

The case where $a_2+a_3=2a_1$ combines certain features of the situations described above. The limiting random saddle point 
sequence
explores all heteroclinic connections, but it makes asymmetric choices determined by its the history, also producing
non-Markov dynamics.

Besides the probability structure of limiting saddle point sequences, our main result also provides a
description of the random times spent by the process at each of the saddles it visits.

\section{Nonrigorous analysis of a linear system}\label{sec:2d-linear-nonrigorous}
In this section we consider the diffusion near a saddle point in the simplest possible case:
\begin{align*}
 dX_{\eps,1}&=\lambda_1 X_{\eps,1}dt+\eps dW_1,\\
 dX_{\eps,2}&=\lambda_2 X_{\eps,2}dt+\eps dW_2
\end{align*}
with initial conditions $X_{\eps,1}=0$, $X_{\eps,2}=1$. Here $\lambda_1>0>\lambda_2$, and $W_1,W_2$ are i.i.d.\
standard Wiener processes.

Let us study the exit distribution of $X_\eps$ for the strip $\{(x_1,x_2): |x_1|\le 1\}$. In other words, we are
interested in the distribution of $X_\eps(\tau_\eps)$, where
$\tau_\eps=\inf\{t:|X_{\eps,1}(t)|=1\}$.
Duhamel's principle implies
\[
 X_{\eps,1}(t)=\eps e^{\lambda_1t}\int_0^t e^{-\lambda_1s} dW_1(s).
\]
The integral in the r.h.s.\ converges a.s.\ to a centered Gaussian r.v.~$N_1$, so that for large $t$,
\[
 |X_{\eps,1}(t)|\approx\eps e^{\lambda_1t} N_1.
\]
Therefore, for the exit time $\tau_\eps$, we have to solve
\[
 1\approx\eps e^{\lambda_1t} |N_1|,
\]
so that
\[
 \tau_\eps\approx\frac{1}{\lambda_1}\ln(\eps |N_1|)^{-1}.
\]
So, the time spent by the diffusion in the neighborhood of the saddle point is about  $\lambda^{-1}\ln(\eps^{-1})$.

On the other hand,
\[
X_{\eps,2}(t)=e^{\lambda_2t} + \eps\int_0^t e^{\lambda_2(t-s)} dW_2(s). 
\]
and the integral in the r.h.s.\ converges in distribution to $N_2$, a centered Gaussian r.v.
Plugging the expression for $\tau_\eps$ into this relation, we get
\[
X_{\eps,2}(\tau_\eps)\approx\eps^{-\lambda_2/\lambda_1}|N_1|^{-\lambda_2/\lambda_1}+\eps N_2
\approx\begin{cases}
        \eps N_2,& \lambda_2<-\lambda_1,\\
        \eps^{-\lambda_2/\lambda_1}|N_1|^{-\lambda_2/\lambda_1},& \lambda_2>-\lambda_1,\\
        \eps(|N_1|+N_2),&  \lambda_2=-\lambda_1.
       \end{cases}
\]

Therefore, when the contraction is stronger than expansion $(\lambda_2<-\lambda_1)$, the limiting exit distribution is
centered Gaussian. In the opposite case $(\lambda_2>-\lambda_1)$, the limiting exit distribution is strongly asymmetric 
and concentrated on the positive semiline reflecting the fact that the initial condition $X_{\eps,2}(0)=1$ was positive
and presenting a strong memory effect. In the intermediate case $(-\lambda_2=\lambda_1)$, the limit is the distribution
of the sum of a symmetric r.v.\ $N_2$ and a positive r.v.\ $|N_1|$, and the resulting asymmetry also serves as a basis
for a certain memory effect.

In general, the asymmetry in the exit distribution means that at the next visited saddle point the choices of the exit direction
will not be symmetric, thus leading to non-Markovian dynamics. Notice that three types of behavior for the linear system that we just
derived correspond to the three types of cycling through the saddle points in the example of
Section~\ref{sec:example}.

One of the main goals of this paper is the precise mathematical meaning of the approximate identities of this section and their generalizations to multiplicative white noise perturbations of nonlinear dynamics in higher dimensions.

\section{Convergence of graphs in space-time}\label{sec:convergence_of_graphs}

The main result of this paper states the convergence of a family of continuous processes to a process with jumps.
This kind of convergence is impossible in the traditionally used Skorokhod topology on the space $D$
of processes with left and right limits, since the set of all
continuous functions is closed in the Skorokhod topology, see~\cite{Billingsley:MR1700749}. 

In this section we replace $D$ by another extension
of the space of continuous functions. This extension allows to describe not only the fact of an instantaneous jump
from one point to another, but also the curve along which the jump is made. We shall also introduce an appropriate topology
to characterize the convergence in this new space.  

It is interesting that in his classical paper~\cite{Skorohod:MR0084897} Skorokhod introduced several topologies for trajectories with jumps. 
Only one of them is widely known as the Skorokhod topology now. However, the construction that we are going to describe here did not appear either
in~\cite{Skorohod:MR0084897}, or anywhere else, to the best of our knowledge, at least in the literature on stochastic processes.

We consider all continuous functions (``paths'') 
\[
\gamma:[0,1]\to[0,\infty)\times\R^d
\]
such that $\gamma^0(s)$ is nondecreasing in $s$.
Here (and often in this paper) we use superscripts
to denote coordinates: $\gamma=(\gamma^0,\gamma^1,\ldots,\gamma^d)$.

We say that two paths $\gamma_1$ and $\gamma_2$
are equivalent, and write $\gamma_1\sim\gamma_2$ if there is a path $\gamma^*$ and nondecreasing surjective functions
$\lambda_1,\lambda_2:[0,1]\to[0,1]$ with 
$\gamma_1(s)=\gamma^*\circ \lambda_1(s)$ and $\gamma_2(s)=\gamma^*\circ\lambda_2(s)$ for all $s\in[0,1]$,
where $\circ$ means the composition of two functions.
(These are essentially reparametrizations of the path $\gamma^*$ except that we allow $\lambda_1,\lambda_2$
to be not strictly monotone.) In 
Section~\ref{sec:basics-on-paths} we shall prove the following statement:

\begin{lemma}\label{lm:equivalence} The relation $\sim$ on paths is a well-defined equivalence relation.
\end{lemma}

Any non-empty class of equivalent paths will be called a curve. 
We denote the set of all curves by~$\X$.
Our choice of the equivalence relation ensures that each curve
is a closed set in sup-norm (see Section~\ref{sec:basics-on-paths}), and we
shall be able to introduce a metric on $\X$ induced by the sup-norm.

Since each curve in $\X$ is nondecreasing in the zeroth coordinate which plays the role of time, it
can be thought of as the graph of a function from~$[0,T]$ to~$\R^d$ for some nonnegative~$T$. However, any value $t\in[0,T]$ can be attained by a path's
``time'' coordinate 
for a whole interval of values of the variable parametrizing the curve, thus defining a curve in $\{t\}\times\R^d$, which is interpreted as the curve along which the jump at time $t$ is made.

We would like to introduce a distance in $\X$ that would be sensitive to the geometry of jump curves, but not to their
parametrization. So, for
two curves $\Gamma_1,\Gamma_2\in\X$, we denote
\begin{equation}
\label{eq:def-distance}
\rho(\Gamma_1,\Gamma_2)=\inf_{\gamma_1\in\Gamma_1, \gamma_2\in\Gamma_2} \sup_{s\in[0,1]}|\gamma_1(s)-\gamma_2(s)|,
\end{equation}
where $|\cdot|$ denotes the Euclidean norm in $[0,T]\times\R^d$.

\begin{theorem}\label{thm:metric_space}
\begin{enumerate}
\item The function $\rho$ defined above is a metric on $\X$.
\item The space $(\X,\rho)$ is Polish (i.e. complete and separable).
\end{enumerate}
\end{theorem}

We postpone the proof of this statement to Section~\ref{sec:basics-on-paths}.

Naturally, any continuous function $f:[0,T]\to\R^d$ defines a path $\gamma_f$ by
\begin{equation}
\label{eq:gamma_f}
\gamma_f(t)=\left(tT,f^1(tT),f^2(tT),\ldots,f^d(tT)\right),\quad t\in[0,1],
\end{equation}
and a curve $\Gamma_f$ that is the equivalence class of $\gamma_f$.

The following result shows that the convergence of continuous functions in sup-norm is consistent
with convergence of the associated curves in metric~$\rho$.

\begin{lemma} \label{lm:graph_coverge_if_continuous_function_converge}
Let $(f_n)_{n\in\N}$ and $g$ be continuous functions on~$[0,T]$ for some $T>0$. Then
\[\sup_{t\in[0,T]}|f_n(t)- g(t)|\to0,\quad n\to\infty\] is necessary and sufficient for 
\[\rho(\Gamma_{f_n},\Gamma_g)\to 0,\quad n\to\infty.\]
\end{lemma}

The proof of this lemma is also given in Section~\ref{sec:basics-on-paths}.  In fact, it can be extended to describe
the convergence of graphs of functions with varying domains.

We shall need the following notion in the statement of our main result.

An element $\Gamma$ of $\X$ is called {\it piecewise constant} if there is a path
$\gamma\in\Gamma$, a number $k\in\N$ and families of
numbers 
\[0=s_0\le s_1\le \ldots\le s_{2k}\le s_{2k+1}=1,\] 
\[0=t_0\le t_1\le\ldots\le t_{k-1}\le t_k
,\]
and points \[y_1,\ldots,y_k\in\R^d,\]
 such that $(\gamma^1(s),\ldots,\gamma^d(s))=y_j$ for  $s\in[s_{2j-1},s_{2j}]$,  $j=1,\ldots,k$, and
$\gamma^0(s)=t_j$  for  $s\in[s_{2j},s_{2j+1}]$, $j=0,\ldots,k$.
A piecewise constant $\Gamma$ describes a particle that sits at point $y_j$ between times
$t_{j-1}$ and $t_j$, and at time $t_j$ jumps 
{\it along the path $\gamma_j=(\gamma^1,\ldots,\gamma^d)|_{[s_{2j},s_{2j+1}]}$}.
It is natural to identify $\Gamma$ with a sequence of points and jumps, and we write
\begin{equation}
\label{eq:piecewise_constant}
\Gamma=(\gamma_0,y_1,\Delta t_1,\gamma_1,y_2,\Delta t_2,\gamma_2,\ldots,y_k,\Delta t_k,\gamma_k),
\end{equation} 
where $\Delta t_j=t_j-t_{j-1}$ denotes the time spent by the particle at point $y_j$.

\section{The setting and the main weak convergence result}
\label{sec:setting}

In this section we describe the setting and state the main result. The conditions of the setting and possible
generalizations are discussed in Section~\ref{sbsec:assumption_discussion}.

We assume that the vector field $b:\R^d\to\R^d$ is $C^2$-smooth, and the $d\times d$-matrix valued function $\sigma$ is
also
$C^2$. We assume that for each $x_0$, the  flow $S^tx_0$ associated to the 
system~\eqref{eq:ode} is well-defined for all $t\in\R$ (including negative values of $t$). We assume that $b$ admits a heteroclinic network of a special kind that
we proceed to describe.

We suppose that there is a finite or countable set of points $(z_i)_{i\in \Cc}$, where $\Cc=\N$ or $\Cc=\{1,\ldots,N\}$
for some $N\ge1$, with the following properties.
\begin{enumerate}
\item[(i)] For each $i\in\Cc$, there is a neighborhood $U_i$ of $z_i$ and a $d\times d$ matrix $A_i$
such that 
\[
b(x)=A_i(x-z_i)+Q_i(x),\quad x\in U_i, 
\]
where $|Q_i(x)|\le C_i|x-z_i|^2$, for a constant $C_i$ and every $x\in U_i$. In particular, $z_i$ is  a critical point for $b$,
since $b(z_i)=0$. 
Moreover, we require that $S^t$ is conjugated on $U_i$
to a linear dynamical system $\dot y=A_iy$ by a $C^2$-diffeomorphism $f_i$ satisfying $f_i(z_i)=0$. This means that 
for any $x_0\in U_i$,
there is $t_0=t_0(x_0)>0$ such that for all $t\in(-t_0,t_0)$,
\[
\frac{d}{dt} f_i(S^tx_0)=A_i f_i(S^tx_0).
\]
\item[(ii)] For each $i\in\Cc$, the eigenvalues $\lambda_{i,1},\ldots,\lambda_{i,d}$ of $A_i$ are real and simple,
we also assume that there is an integer $\nu_i$ with $2\le\nu_i\le d$ such that
\begin{equation}
\lambda_{i,1} >\ldots> \lambda_{i,\nu_i-1}>0>\lambda_{i,\nu_i}>\ldots>\lambda_{i,d}.
\label{eq:lambdas-i}
\end{equation}
\end{enumerate}

These requirements mean, in particular, that each $z_i$ is a hyperbolic fixed point (saddle) for the
dynamics. The Hartman--Grobman theorem (see Theorem~6.3.1 in~\cite{Katok--Hasselblatt}) guarantees the existence of a
homeomorphism conjugating the flow
generated by vector field $b$ to linear dynamics. Our condition (i) imposes a stronger requirement for this conjugation
to be $C^2$. This requirement is still often satisfied as follows from the Sternberg linearization theorem for
hyperbolic fixed points with no resonances, see 
Theorem~6.6.6 in~\cite{Katok--Hasselblatt}. In particular, the cubic system of Section~\ref{sec:example} is
$C^\infty$-conjugated to a linear system at each saddle point for typical values of its parameters.

We also want
to make several assumptions on orbits of the flow connecting these saddle points to each other.
First, we denote by $v_{i,1},\ldots,v_{i,d}$ the unit eigenvectors associated with the eigenvalues $\lambda_{i,1},\ldots,\lambda_{i,d}$. 
The hyperbolicity (and, even more straightforwardly, the conjugation to a linear flow) implies that
for every $i\in\Cc$ there is a $d-\nu+1$-dimensional $C^2$-manifold $\Wc_i^s$
containing $z_i$ such that $\lim_{t\to+\infty}S^tx= z_i$ for every $x\in\Wc_i^s$  (i.e.
$\Wc_i^s$ is the stable manifold associated to $z_i$.) For each $i$ the
unstable manifold of $z_i$ is also well-defined.
 
However, it is known, see \cite{Kifer} and \cite{Bakhtin-SPA:MR2411523}, that if the initial data for the stochastic
flow are close to
the stable manifold then after passing the saddle $z_i$ the solution evolves mostly along the invariant manifold
associated to the highest eigenvalue of $A_i$. So, what we  need is
the curve $\gamma_i\in C^2$ containing $z_i$,
tangent to $v_{i,1}$ at $z_i$, and invariant under the flow. Of course, the intersection of $\gamma_i$ 
with $U_i$ is well-defined and coincides with 
$f_i^{-1}(\Span(v_{i,1})\cap f_i(U_i))$.

For each $i\in\Cc$ we denote $g_i=f_i^{-1}$ and set
\[
q_i^{\pm}=g_i(\pm R_iv_{i,1}),
\] 
where the numbers $(R_i)_{i\in\Cc}$ are chosen so that 
\[
\tilde U_i=\left\{x:\ \max_{k=1,\ldots,d} |P^{i,k}f_i(x)|\le R_i\right\}\subset U_i,
\]
and these sets are mutually disjoint.
Here, for any $y\in\R^d$,  the number $P^{i,k}y$ is defined by
\[
y=\sum_{k=1}^d (P^{i,k}y) v_{i,k},
\]
and denotes the $k$-th coordinate of $y$ in the coordinate system defined by $v_{i,1},\ldots,v_{i,d}$.

 We denote the
orbits of $q^\pm_i$ by $\gamma_i^\pm$, and assume that for each~$i\in\Cc$, there are numbers~$n^\pm(i)\in\Cc$ such that
\[
\lim_{t\to\infty} S^tq^\pm_i=z_{n^\pm(i)}.
\]
This means that the curves $\gamma^{\pm}_i$ are heteroclinic orbits connecting the saddle point $z_i$ to
saddle points $z_{n^\pm(i)}$.
We do not prohibit these orbits to be homoclinic and connect $z_i$ to itself, i.e. the situations where $n^\pm(i)=i$ are allowed. 

For any $i\in\Cc$ we define 
\begin{align}
h^\pm_i&=\inf\{t:\ S^tq^\pm_i\in \tilde U_{n^\pm(i)}\},\label{eq:hpm}\\
x^\pm_i&=S^{h^\pm_i}q^\pm_i.\label{eq:xpm}
\end{align}
so that $h^\pm_i$ is the time it takes to travel from $q^\pm_i$ to the neighborhood of the next saddle, and $x^\pm_i$
is the point of entrance to that neighborhood.

Our first nondegeneracy assumption is that for all $i\in\Cc$, 
\begin{equation}
\label{eq:nondegeneracy_in_nu_coordinate}
P^{n^\pm(i),\nu_{n^\pm(i)}} f_{n^\pm(i)}(x^\pm_i)\ne 0,
\end{equation}
which means that each heteroclinic orbit $\gamma^{\pm}_i$ has a nontrivial component in the direction of
the $v_{n^\pm(i),\nu_{n^\pm(i)}}$ as  it approaches $z_{n^\pm(i)}$. Although this condition holds true for
all systems of interest (e.g., the system considered in Section~\ref{sec:example}), it is easy to adapt our reasoning to the situations where other components of
the projection of $f_{n^\pm(i)}(x^\pm_i)$ on the stable directions dominate.

We shall also need a nondegeneracy condition on the linearization of~\eqref{eq:ode} along $\gamma^{\pm}_i$. 
For each $x\in\R^d$ we consider the fundamental
matrix $\Phi_x(\cdot)$ solving the equation in variations along the orbit $(S^tx)_{t\ge 0}$:
\begin{align*}
\frac{d}{dt}\Phi_x(t)&=Db(S^tx)\Phi_x(t),\quad t\ge0,\\
\Phi_x(0)&=I.
\end{align*}
For all $i,j$ we denote 
\[
\bar v^{\pm}_{i,j}= \Phi_{q^{\pm}_i}(h_i^\pm) (Df_i(q_i^\pm))^{-1}v_{i,j}.
\]
The technical nondegeneracy assumption on $\Phi$ that we need is:
\begin{multline}
\label{eq:nondegeneracy_in_other_coordinates}
P^{n^\pm(i),k} Df_{n^\pm(i)}(x^\pm_i) \bar v^\pm_{i,j}\ne 0,\quad i\in\Cc,\ j\in\{2,\nu_i\},\ 
k=\begin{cases}1,2,&\nu_{n^\pm(i)}>2,\\
1,&\nu_{n^\pm(i)}=1.\end{cases}
\end{multline}
Again, we work with this condition since it holds true for any system of interest, but
it is easy to adapt our reasoning to the situations where
this condition is violated.

To formulate our main theorem we need a notion of an {\it entrance-exit map} describing
the limiting behavior of $X_\eps$ in the neighborhood of a saddle points, namely, the asymptotics
of the random entrance-exit Poincar\'e map as $\eps\to0$.

We denote  the set of all probability Borel measures on $\R^d$ by $\Pc(\R^d)$.

Let us denote by $\In_i$ the set of all triples $(x,\alpha,\mu)$ where
\begin{enumerate}
\item
$x\in\tilde U_i\cap \Wc^s_i$ 
satisfies $P^{\nu_i}_i f_i(x)\ne 0$;
\item $\alpha\in(0,1]$;
\item $\mu\in\Pc(\R^d)$ with
\[
\mu\{\phi:\ P^{1}_i Df_i(x)\phi\ne 0\}=1,\quad\text{if}\ \alpha<1,
\]
\[
\mu\{\phi:\ P^{2}_i Df_i(x)\phi\ne 0\}=1,\quad\text{if}\ \alpha<1\ \text{and}\ \nu_i>2.
\]
\end{enumerate}
This set will be used to describe the initial condition for equation~\eqref{eq:sde}:
$ X_\eps(0)=x+\eps^\alpha\phi_\eps,$
where the distribution of $\phi_\eps$ weakly converges to $\mu$ as $\eps\to 0$.

We also define
\[
\Out=\{(t,p,x,\beta, F):\ t\in(0,\infty),\  p\in[0,1],\ x\in\R^d,\ \beta\in(0,1], F\in \Pc(\R^d)\}
\]
and
\begin{multline*}
\widehat \Out_i=\{((t_-,p_-,x_-,\beta_-, F_-),(t_+,p_+,x_+,\beta_+, F_+))\in\Out^2:\\ t_-=t_+,\ x_\pm=x_{i}^{\pm},\ p_-+p_+=1,\  \beta_-=\beta_+\}.
\end{multline*}

Here, the numbers $p_\pm$ define the limiting probabilities of choosing each of the two branches of the invariant curve
associated with the highest eigenvalue of the linearization at the saddle point; $x_\pm$ are points on these orbits
that serve as entrance points to neighborhoods of the next saddle points; $t_\pm$ are the times it takes to reach
these
points under the proper (logarithmic) renormalization; $\beta$ is the scaling exponent so that the exit
distribution (serving as the entrance
distribution to the next saddle's neighborhood) takes the form $x_\pm+\eps^\beta\psi_\eps$, where the distribution of
$\psi_\eps$ converges to~$F_+$ or~$F_-$ depending on which of the two branches was chosen.

It is possible (see Lemma~\ref{lm:iteration_lemma}) to give a precise description of the asymptotic behavior of the
diffusion in the
neighborhood of each saddle point in terms of an appropriate entrance-exit map, i.e.\ a map that for each saddle point
computes a description of the exit parameters in terms of the entrance parameters: 
\[
\Psi_i:\In_i\to \widehat\Out_i,\quad i\in\Cc.
\]

For an entrance-exit map $(\Psi_i)_{i\in\Cc}=(\Psi_{i,-},\Psi_{i,+})_{i\in\Cc}$ we shall denote
its components by $t_i=t_{i,\pm}$, $p_{i,\pm}$, $x_{i,\pm}$, $\beta_i=\beta_{i,\pm}$,
$F_{\pm,i}$.

\medskip

Suppose $x_0$ belongs to one of heteroclinic orbits of the network. 
A sequence $\z=(\theta_0, z_{i_1},\theta_{1},z_{i_2},\theta_{2}\ldots,\theta_{k-1},z_{i_k},\theta_k)$
is called {\it admissible} for $x_0$ if 
\begin{enumerate}
\item $\theta_0$ is the positive orbit of $x_0$ with $\lim_{t\to\infty}S^tx_0=z_{i_1}$;
\item for each $j\in\{1,\ldots,k\}$, 
$\theta_{j}=\gamma_{i_j}^+$ or $\theta_{j}=\gamma_{i_j}^-$;
\item for each $j\in\{1,\ldots,k-1\},$ 
\[
i_{j+1}=
\begin{cases}
n^+(i_j),& \theta_{j}=\gamma_{i_j}^+,\\
n^-(i_j),& \theta_{j}=\gamma_{i_j}^-.
\end{cases}
\]
\end{enumerate}
The number $k=k(\z)$ is called the length of $\z$.

Our main limit theorem uses entrance-exit maps to assign limiting probabilities to admissible sequences. Let us proceed
to describe this procedure.

With each admissible sequence $\z$ we associate the following sequence:
\begin{equation}
\label{eq:associated_sequence}
\zeta(\z)=((\tilde x_0,\alpha_0,\mu_0),(t_1,p_1,\tilde x_1,\alpha_1,\mu_1),\ldots,(t_k,p_k,\tilde x_k,\alpha_k,\mu_k)).
\end{equation}
Here $\tilde x_0= S^{\tilde t(x_0)}x_0$, 
$\alpha_0=1$, and
\begin{equation}
\label{eq:mu_0}
\mu_0=\Law\left(\Phi_{x_0}(\tilde t(x_0))\int_0^{\tilde t(x_0)}\Phi_{x_0}^{-1}(s)\sigma(S^sx_0)dW(s)\right),
\end{equation}
where 
\begin{equation}
\tilde t(x_0)=\inf\{t\ge 0:\ S^tx_0\in \tilde U_{i_1}\}+1.
\label{eq:tilde_t}
\end{equation}
We add 1 in the r.h.s.\ so that the distribution $\mu_0$ is nondegenerate (and the maps
$\Psi_{i_1,\pm}(x_0,\alpha_0,\mu_0)$ are well-defined) even if $x_0\in\tilde U_{i_1}$.
All other entries in~\eqref{eq:associated_sequence} are obtained according to the following recursive procedure. 
For each~$j$, 
\begin{equation}
\label{eq:iteration}
(t_j,p_j,\tilde x_j,\alpha_j,\mu_j)
=\begin{cases}\Psi_{i_j,+}(\tilde x_{j-1},\alpha_{j-1},\mu_{j-1}),&\theta_{j}=\gamma_{i_j}^+\\
               \Psi_{i_j,-}(\tilde x_{j-1},\alpha_{j-1},\mu_{j-1}),&\theta_{j}=\gamma_{i_j}^-.
      \end{cases}
\end{equation}
The numbers $t_1=t_1(\z),\ldots,t_k=t_k(\z)$ defined above play the role of time, and the
admissible sequence $\z$ can be identified with a  piecewise constant trajectory $\Gamma(\z)\in\X$: 
\[
\Gamma(\z)=(\theta_0,z_{i_1},t_1,\theta_1,z_{i_2},t_2,\theta_2,\ldots,z_{i_k},t_k,\theta_{k}).
\]

The numbers $p_1=p_1(\z),\ldots,p_k=p_k(\z)$ defined in \eqref{eq:associated_sequence} play the roles of conditional probabilities,
and we denote
\begin{equation}
\label{eq:product_of_conditional_probabilities}
\pi(\z)=p_1(\z)p_2(\z)\ldots p_{k}(\z).
\end{equation}

The set of all admissible sequences for $x_0$ 
has the structure of a binary tree. The natural partial order on it is determined by inclusion.
 We say that
a set $L$ of admissible sequences for $x_0$ is {\it free}
 if no two sequences in $L$ are comparable with respect to this
partial order. If additionally, for any sequence not from $L$, it is comparable to one of the sequences from~$L$,
the set $L$ is called {\it complete}.
In the language of graph theory, a complete set is a section of the binary tree. 

It is clear that for any free set $L$, $\pi(L)\le 1$, where $\pi(L)=\sum_{\z\in L}\pi(\z)$.
A free set $L$ is called {\it conservative} if 
$\pi(L)= 1$.
Every complete set is finite and conservative.



We are ready to state our main result now.
\begin{theorem}\label{th:main}  Under the conditions stated above there is an
entrance-exit map $\Psi$ with the following property.

Let  $x_0$ belong to one of the heteroclinic orbits
of the network. For each $\eps>0$ define a stochastic process $Z_\eps$ by 
\begin{equation}
\label{eq:Z_eps_is_rescaled_X_eps}
Z_\eps(t)=X_\eps(t\ln(\eps^{-1})),\quad t\ge 0,
\end{equation}
where $X_\eps$ is the strong solution of~\eqref{eq:ode} with initial condition $X_\eps(0)=x_0$.

For any conservative set $L$ of $x_0$-admissible sequences, there is a family of stopping times $(T_\eps)_{\eps>0}$ such that
the distribution of the graph
$\Gamma_{Z_\eps(t),t\le T_\eps}$ converges weakly in $(\X,\rho)$ to the measure $M_{x_0,L}$ 
concentrated on the set
\[
\{\Gamma(\z)\ :\ \z\in L\},
\]
and satisfying
\begin{equation}
\label{eq:def_of_M}
M_{x_0,L}\{\Gamma(\z)\}=\pi(\z),\quad \z\in L,
\end{equation}
where $\pi(\z)$ is defined via $\Psi$ in~\eqref{eq:product_of_conditional_probabilities}.
\end{theorem}

For any entrance-exit map the family of conservative sets includes all finite complete sets, so that the content of Theorem~\ref{th:main} is nontrivial.

Importantly, we actually construct the desired entrance-exit map $\Psi$ in the proof. This allows to study the details
of the limiting process in Section~\ref{sec:discussion}. At this point let us just mention that 
the sequence of saddles visited by the limiting process can be Markov or non-Markov depending
on the eigenvalues of the linearizations at the saddle points.
We discuss this memory effect and some other 
implications and possible extensions of  Theorem~\ref{th:main} in Section~\ref{sec:discussion}.

\section{Exit measure asymptotics}\label{sec:exit}

We shall now apply Theorem~\ref{th:main} to the exit problem along a heteroclinic network, and formulate a theorem that, in a sense, gives more precise information on the exit distribution than the FW theory.
We assume that there is a domain $D\subset\R^d$ with piecewise smooth boundary
such that $x_0\in D$. The FW theory implies that, as $\eps\to 0$, the exit
measure for the process~$X_\eps$ started at $x_0$ concentrates at points $y\in\partial D$ that provide the minimum value
of the so called quasi-potential $V(x_0,y)$. Since for all the points that are reachable
from $x_0$ along the heteroclinic network, the quasi-potential equals~0, we conclude
that in the case of heteroclinic networks, the exit measure
asymptotically concentrates at the boundary points that can be
reached from $x_0$ along the heteroclinic network. However, this approach does not allow to
distinguish between the exit points while ours allows to determine an exact limiting probability for each
exit point.

We take a point $x_0$ on one of the heteroclinic orbits of the network and denote by $L(x_0,D)$ the set of 
all the admissible sequences 
$\z$ (of any length $k$) for $x_0$ such that the last curve $\theta_{k}$ of the sequence intersects
$\partial D$ transversally at a point $q(\z)$ (if there are several points of intersection we take
the first one with respect to the natural order on $\theta_{k}$), and $\theta_j$ does not intersect $\partial D$ for
all $j<k$.   
Let 
\[
\tau_{\eps}(x_0,D)=\inf\{t:X_\eps(t)\in\partial D\}.
\]
The distribution of $X_\eps(\tau_\eps(x_0,D))$ is concentrated on $\partial D$.

For each $\z\in L(x_0,D)$, one can define $\pi(\z)$ via \eqref{eq:product_of_conditional_probabilities}.
\begin{theorem}\label{th:exit_asymptotics} For the setting described above, if the set $L(x_0,D)$
is conservative,
then the distribution of $X(\tau_\eps(x_0,D))$ converges
weakly, as $\eps\to0$,  to 
\[
P_{x_0,D}=\sum_{\z\in L(x_0,D)}\pi(\z)\delta_{q(\z)}.
\]
\end{theorem}
\begin{proof} Let us use Theorem~\ref{th:main} to choose the times
$(T_\eps)_{\eps>0}$ providing convergence of the distribution of $\Gamma_{Z_\eps(t),t\le T_\eps}$
to $M_{x_0,L(x_0,D)}$. The theorem follows since
$X_\eps(\tau_\eps(x_0,D))$ is a functional of $\Gamma_{Z_\eps(t),t\le T_\eps}$, continuous on the support
of $M_{x_0,L(x_0,D)}$. 
\end{proof}

\begin{remark}\rm
Notice that for different sequences
$\z$ and $\z'$ it is possible to have $q(\z)=q(\z')$ so that an exit point can accumulate its
limiting probability from a variety of admissible sequences.
\end{remark}
\begin{remark}\rm
The behavior of the system up to $\tau_{\eps}(x_0,D)$ is entirely determined by the drift and
diffusion coefficients inside $D$. Therefore, there is an obvious generalization of this theorem for
heteroclinic  networks in a domain, where one requires the invariant manifolds associated to the highest
eigenvalue at a critical point to connect that critical point either to another critical point, or to a point on $\partial D$. An advantage of that theorem is that one does not have to specify the (irrelevant) coefficients of~\eqref{eq:sde} outside of~$D$. We omit the precise formulation for brevity. 
\end{remark}
\begin{remark}\rm 
In the case of nonconservative $L(x_0,D)$, the limit theorem is harder to formulate. In the limit, the exit happens
along the sequences belonging to $L=L(x_0,D)$ with positive probability $\pi(L)<1$. With
probability $1-\pi(L)$, the exit happens in a more complicated way (and in a longer than logarithmic time) depending on the details of the driving vector
field. 
\end{remark}

\section{Proof of Theorem~\ref{th:main}}
\label{sec:main_proof}

We begin with the central lemma that we need in the proof. It has a lengthy statement, and after formulating it, we
also give a brief informal explanation.

For each $i\in\Cc$ we introduce $U^+_i$ and $U^-_i$ via
\begin{equation}
U^\pm_i=\{x\in U: P^{i,1}f_i(x)=\pm R_i\}.
\label{eq:exit_boundary}
\end{equation}


\begin{lemma}\label{lm:iteration_lemma} For each $i\in\Cc$, there is a map
\[
\Psi_i=((t_{i,-},p_{i,-},x_{i,-},\beta_{i,-},
F_{i,-}),(t_{i,+},p_{i,+},x_{i,+},\beta_{i,+}, F_{i,+})):\In_i\to \widehat\Out_i
\]
with the following property. 

Take any $(x,\alpha,\mu)\in \In_i$
and any family of distributions $(\mu_\eps)_{\eps>0}$ in $\Pc(\R^d)$ with
$\mu_\eps\Rightarrow \mu$ as $\eps\to0$.
For each $\eps>0$, consider the solution $X_\eps$ of~\eqref{eq:sde} with initial condition
\begin{equation}
X_\eps(0)=x+\eps^\alpha\phi_\eps,
\label{eq:inidata_before_splitting} 
\end{equation}
where
\begin{equation}
\label{eq:convergence_in_tangent_space}
\Law(\phi_\eps)= \mu_\eps,\quad \eps>0,
\end{equation}
and define a stopping time
\[
T_{\out,\eps}=\inf\{t\ge0: X_\eps(t)\in U_i^+\cup U_i^-\},\quad \eps>0,
\]
and two events
\[
A_{i,\pm,\eps}=\{X_\eps(T_{\out,\eps})\in U^\pm_i\}, \quad \eps>0.
\]
Then
\begin{enumerate}
\item\label{it:convergence_of_times} As $\eps\to0$,
\[\frac{T_{\out,\eps}}{\ln(\eps^{-1})}\stackrel{\Pp}{\to}
t_{i,\pm}(x,\alpha,\mu).\]
\item\label{it:convergence_of_splitting_probabilities} As $\eps\to0$,
\[
\Pp(A_{i,\pm,\eps})\to p_{i,\pm}(x,\alpha,\mu).
\]
\item\label{it:distribution_in_the_tangent_space}  
There is a family of  random vectors
$(\psi_{i,\eps})_{\eps>0}$ such that
\begin{enumerate}
\item for every $\eps>0$, on
$A_{i,\pm,\eps}$
\[X_\eps(T_{\out,\eps}+h^\pm_{i})=x_{i,\pm}(x,\alpha,\mu)+\eps^{\beta_{i,\pm}(x,\alpha,\mu)} \psi_{i,\eps}, \]
where $h^{\pm}_i$ was defined in~\eqref{eq:hpm};
\item as $\eps\to0$,
\[\Law(\psi_{i,\eps}| A_{i,\pm,\eps})\Rightarrow F_{i,\pm}(x,\alpha,\mu);
\]
\item
\[
F_{i,\pm}(x,\alpha,\mu)\left\{\psi: P^{n^\pm(i),k} Df_{n^{\pm}(i)}(x_{i,\pm})\psi\ne 0\right\}=1,\quad k=
\begin{cases}
1,2,&\nu>2,\\1,&\nu=2.
\end{cases}
\]
\end{enumerate}
\item \label{it:tracking} For any $r>0$, there is $T(r)$ such that, as $\eps\to0$,
\[
\Pp\left\{\sup_{0\le t\le T(r)}|X_\eps(t)-S^tx|\ge r\right\}\to 0,
\]
\[
\Pp\left\{\sup_{T(r)\le t\le T_{\out,\eps}-T(r)}|X_\eps(t)-z_{i}|\ge
r\right\}\to 0,
\]
\[
\Pp\left(A_{i,\pm,\eps}\cap
\left\{\sup_{T_{\out,\eps}-T(r)\le t\le T_{\out,\eps}+h^\pm_{i}}|X_\eps(t)-S^{t-T_{\out,\eps}}q_{i,\pm}|\ge
r\right\}\right)\to 0.
\]
\item\label{it:sublog-tracking} For any $r>0$, as $\eps\to0$,
\[
\Pp\left(A_{i,\pm,\eps}\cap
\left\{\sup_{0\le t\le (\ln\eps^{-1})^{1/2}} |X_\eps(T_{\out,\eps}+h^\pm_{i}+t)-S^t x_{i,\pm}|\ge
r\right\}\right)\to 0.
\]
\end{enumerate}
\end{lemma}

\medskip
Part~\eqref{it:convergence_of_times} of the lemma describes the asymptotic behavior of exit times.
Part~\eqref{it:convergence_of_splitting_probabilities} determines the limiting probabilities of choosing each of the two
outgoing heteroclinic orbits. Part \eqref{it:distribution_in_the_tangent_space} describes the entrance distribution for the
next visited saddle point (it takes $T_{\out,\eps}+h^\pm_{i}$ to reach its neighborhood); parts (a) and (b)
give the asymptotic scaling law, and part (c) ensures that
we can apply this lemma at the next saddle, too, which gives rise to the iteration scheme (see the definition of $\In_i$).
Part~\eqref{it:tracking} formalizes the fact that for small $\eps$, with high probability, the diffusion trajectory 
first closely follows the deterministic trajectory $S^tx$, then spends some time in a small neighborhood of the saddle point,
and then follows closely one of the outgoing heteroclinic connections until it reaches the neighborhood of the next saddle point.
Part~\eqref{it:sublog-tracking} shows that after reaching the neighborhood of the next saddle point the trajectory
continues to follow the same heteroclinic connection for a sublogarithmic time.

The proof of this Lemma will be given in Sections~\ref{sec:one_saddle} and \ref{sec:along_heteroclinic}. 
The solution~$X_\eps$ spends most of the time in the neighborhood of the saddle points and in beetween it travels from
one saddle point to another along a heteroclinic connection. 
We split the analysis in two parts accordingly. In Section~\ref{sec:one_saddle} we describe
the behavior of the system in the neighborhood of the saddle point assuming that the initial data is given
by~\eqref{eq:inidata_before_splitting}. In Section~\ref{sec:along_heteroclinic} we describe the motion between
neighborhoods
of two saddle points and finish the proof of Lemma~\ref{lm:iteration_lemma}.

The rest of this section is devoted to the derivation of our main result from Lemma~\ref{lm:iteration_lemma}.

\begin{proof}[Proof of Theorem~\ref{th:main}.] We have to show that the map $\Psi$ constructed in 
Lemma~\ref{lm:iteration_lemma} satisfies the statement of the theorem. 

For any sequence $\z=(\theta_0, z_{i_1},\theta_{1},z_{i_2},\theta_{2}\ldots,\theta_{k-1},z_{i_k},\theta_k)$
and any $\eps$ we define a sequence of stopping times in the following way. First, we set
\[
\tau_{\eps,\z,1,\inin}=\tilde t(x_0),
\]
where $\tilde t(x_0)$ was defined in \eqref{eq:tilde_t}. Then, for $j=1,\ldots,k$, we recall that $U_{i_j}^{\pm}$
was defined in~\eqref{eq:exit_boundary} and define
recursively
\begin{align*}
\tau_{\eps,\z,j,\out}&=
\begin{cases}
\inf\{t\ge \tau_{\eps,\z,j,\inin} :\ X_\eps(t)\in  U^+_{i_j}\},&\theta_j=\gamma^+_{i_j},\\
\inf\{t\ge \tau_{\eps,\z,j,\inin} :\ X_\eps(t)\in  U^-_{i_j}\},&\theta_j=\gamma^-_{i_j},
\end{cases}
\\
\tau_{\eps,\z,j+1,\inin}&=
\begin{cases}
\tau_{\eps,\z,j,\out}+h_{i_j}^+,& \theta_j=\gamma^+_{i_j},\\
\tau_{\eps,\z,j,\out}+h_{i_j}^-,& \theta_j=\gamma^-_{i_j}.
\end{cases}
\end{align*}
Less formally, for each $j$,  the process $X_\eps$ leaves the neighborhood
of $z_{i_j}$ at time $\tau_{\eps,\z,j,\out}$ and travels for time $t_{i_j}^\pm$ along one of the two heteroclinic connections $\theta_j=\gamma^\pm_{i_j}$ emerging from
$z_{i_j}$. At time $\tau_{\eps,\z,j+1,\inin}$ the solution is in the neighborhood of the next saddle
point of the sequence, close to $x^{\pm}_{i_j}$. 

The last point of the sequence $\z$ is special. We define now 
\[
T_{\eps,\z}=\frac{\tau_{\eps,\z,k+1,\inin}+(\ln\eps^{-1})^{1/2}}{\ln(\eps^{-1})}
\] 
and 
\[
 T_\eps=\eps^{-1}  \wedge \inf_{z\in L}T_{\eps,\z}.
\]
The $\eps^{-1}$ term is introduced to make $T_\eps$ finite (to take into account the improbable case where $X_\eps$
does not evolve along
any $\z\in L$) so that the curve
$\Gamma_{Z_\eps(t),t\le T_\eps}$ is well-defined.

According to Theorem~2.1 in \cite{Billingsley:MR1700749},
it suffices to show that for any open set $G\in\X$,
\[\liminf_{\eps\to 0} \Pp\{\Gamma_{Z_\eps(t),t\le T_\eps}\in G\}\ge M_{x_0,L}(G),\]
where the probability measure $M_{x_0,L}$ is defined by~\eqref{eq:def_of_M}.
Since the
measure~$M_{x_0,L}$ is discrete, it is sufficient to check that for any
$\z\in L$ and any open set
$G\subset \X$ containing $\Gamma(\z)$,
\begin{equation}
\lim_{\eps\to0}\Pp\{\Gamma_{Z_\eps(t),t\le T_\eps}\in G\}\ge\pi(\z).
\label{eq:limit_probability}
\end{equation}

We start with the linearization along the orbit $S^tx_0$. It follows 
from~\cite{Blagoveschenskii:MR0139204} (see also Lemma~\ref{lm:linearization_on_finite_interval}) that 
\[X_\eps(\tilde t(x_0))=\tilde x_0+\eps\phi_\eps,\]
with $\Law(\phi_\eps)\Rightarrow\mu_0,$
where $\mu_0$ was defined in~\eqref{eq:mu_0}.

Now the strong Markov property allows to apply  Lemma~\ref{lm:iteration_lemma} iteratively along the saddle points of
sequence $\z$. We can complete the proof by reparametrizing the heteroclinic connections of $\z$
appropriately and applying the following proximity criterion to
derive~\eqref{eq:limit_probability}.

\begin{lemma} Suppose that $\delta$ is a positive number and a path $\gamma$ defines  a piecewise constant curve
$\Gamma$ as given by~\eqref{eq:piecewise_constant}. Suppose a continuous function~$f$ and  nondecreasing nonnegative
number sequences $(r_m)_{m=0}^{2k+1}$ and $(s'_m)_{m=0}^{2k+1}$
satisfy the following properties:
\begin{enumerate}
\item $t_j-\delta\le r_{2j}\le t_j\le r_{2j+1}\le t_j+\delta$,\quad $j=0,\ldots,k$;
\item $|f(r)-y_j|\le \delta$\quad for $j=1,\ldots,k$ and $r\in[r_{2j-1},r_{2j}]$; 
\item $s_{2j}\le s'_{2j}\le s'_{2j+1}\le s_{2j+1}$,\quad $j=0,\ldots,k$;
\item for each $j=0,\ldots,k$, there is a nondecreasing bijection \[\lambda_j:[r_{2j},r_{2j+1}]\to[s'_{2j},s'_{2j+1}]\]
such that \[|\gamma(\lambda_j(r))-(r,f^1(r),\ldots, f^d(r))|<\delta,\quad r\in[r_{2j},r_{2j+1}];\]   
\item for each $j=0,\ldots,k$,
\begin{align*}
|\gamma(s)-\gamma(s_{2j})|\le\delta,&\quad s\in[s_{2j},s'_{2j}],\\
|\gamma(s)-\gamma(s_{2j+1})|\le\delta,&\quad s\in[s'_{2j+1},s_{2j+1}].
\end{align*}
\end{enumerate}
Then $\rho(\Gamma,\Gamma_f)\le 3\delta$.
\label{lm:proximity_to_piecewise_constant}
\end{lemma}
The proof of Lemma~\ref{lm:proximity_to_piecewise_constant} is given in Section
~\ref{sec:basics-on-paths}
\end{proof}


\section{The system in the neighborhood of a saddle point}\label{sec:one_saddle}

In this section we fix $i\in\Cc$ and consider a saddle point $z_i$. We recall our assumption that the dynamics generated
by $b$ in a small
neighborhood $U_i$ of $z_i$ is conjugated by a $C^2$-diffeomorphism $f_i$ to that generated by a linear vector 
field generated by a matrix $A_i$ in a neighborhood
of the origin.  In this section we often denote $A_i,f_i$, etc.\ by
$A,f$, etc., omitting the dependence on $i$. In particular, the flow 
generated by the linearized vector field is given by $e^{tA}$ and sometimes
will be denoted by $S^t_A$.

We recall that the eigenvalues $\lambda_1,\ldots,\lambda_d$ of $A$ are real and simple,
and there is a number $\nu\ge2$ such that
\[
\lambda_1 >\ldots> \lambda_{\nu-1}>0>\lambda_{\nu}>\ldots>\lambda_d.
\]
We denote the associated eigenvectors by $v_1,\ldots,v_d$, and introduce the coordinates $u^1,\ldots, u^d$ of a vector $u$ via $u=\sum_{k=1}^d u^k v_k$. We define
\[
L=\Span\{v_2,\ldots,v_d\},\quad L^-=\Span\{v_\nu,\ldots,v_d\},
\]
and denote by $\Pi_L$ the projection on $L$ along $v_1$.

We begin the analysis with the derivation of
the It\^o equation for $Y_\eps(t)=f(X_\eps(t))$. 
 The It\^o formula gives:
\begin{align*}
dY_{\eps}^j(t)&=\sum_{k=1}^d \partial_k f^j(X_\eps(t))dX_{\eps}^k
+\frac{\eps^2}{2}\sum_{k,l=1}^d(\sigma\sigma^*)^{kl}(X_\eps(t)) \partial_{kl} f^j(X_\eps(t)) dt\\
   &=\sum_{k=1}^d \partial_k f^j(X_\eps(t))b^k(X_\eps(t))dt
     + \eps\sum_{k=1}^d \partial_k f^j(X_\eps(t))\sigma^k_m(X_\eps(t))dW^m(t)\\
&+\frac{\eps^2}{2}\sum_{k,l=1}^d(\sigma\sigma^*)^{kl}(X_\eps(t)) \partial_{kl} f^j(X_\eps(t)) dt
\end{align*}
where $\sigma^*$ denotes the transpose of $\sigma$, and $(\sigma\sigma^*)^{kl}(x)$ denotes $(\sigma(x)\sigma^*(x)))^{kl}$. 
Since the pushforward of the vector field $b$ under $f$ at
a point $y$ is given by~$Ay$, we have
\begin{equation*}
Ay= Df(g(y))b(g(y)),
\end{equation*}
where $g=f^{-1}$, and the equation above rewrites as
\begin{align*}
dY_{\eps}^j(t)   &=\sum_{k=1}^d A^j_kY_\eps^k(t) dt
     + \eps\sum_{k=1}^d \partial_k f^j(g(Y_\eps(t)))\sigma^k_m(g(Y_\eps(t)))dW^m(t)\\
&+\frac{\eps^2}{2}\sum_{k,l=1}^d(\sigma\sigma^*)^{kl}(g(Y_\eps(t))) \partial_{kl} f^j(g(Y_\eps(t))) dt.
\end{align*}

We rewrite the last identity as
\begin{equation}
dY_{\eps}(t)=AY_\eps(t)dt + \eps B(Y_\eps(t))dW(t) + \eps^2 C(Y_\eps(t))dt,
\label{eq:main_conjugated_system}
\end{equation}
or, equivalently,
\begin{align*}
dY^j_{\eps}(t)&=A^j_kY^k_\eps(t)dt + \eps B^{j}(Y_\eps(t))dW(t) + \eps^2 C^j(Y_\eps(t))dt\\
&=A^j_kY^k_\eps(t)dt + \eps\sum_{k=1}^d B^{j}_k(Y_\eps(t))dW^k(t) + \eps^2 C^j(Y_\eps(t))dt.
\end{align*}
We see that $B$ and $C$ are continuous and bounded in the neighborhood $f(U)$ of the origin,
and $B$ is nodegenerate.

It is also clear that if we assume~\eqref{eq:inidata_before_splitting}, then
\begin{equation}
Y_\eps(0)=y_0+\eps^\alpha \xi_\eps,
\label{eq:xi_0}
\end{equation}
where $y_0=f(x)$, and
$\xi_\eps$ converges, as $\eps\to0$, in distribution to $\xi_0=DF(x_0)\phi_0$, with $\Law(\phi_0)=\mu$, see \eqref{eq:convergence_in_tangent_space}.



Recall that for the saddle point $z=z_i$, two neighborhoods $\tilde U\subset U$ are defined.  
Define $V=f(U)$ and $\tilde V=f(\tilde U)$, so that 
\[
\tilde V=\{y\in\R^d:\ |y^j|\le R, j=1,\ldots,d\}\subset V,
\]
(we use the notation  $R=R_i$, $\tilde U=\tilde U_i$, $U=U_i$ in this section).
Then $y_0\in \tilde V\cap L^-$. In particular, $y_0^k=0$ for all $k<\nu$.

In the remainder of this section we study the system~\eqref{eq:main_conjugated_system}
 with initial data given by \eqref{eq:xi_0}. The solution is actually defined up to a stopping time~$t_{V,\eps}$ at
which the solution hits $\partial V$. Let us define another stopping time
\[
t_\eps=\inf\{t\ge0:\ |Y^1_\eps(t)|=R\}\wedge t_{V,\eps}.
\]
As we shall see later,
\[
\Pp\{t_\eps <  t_{V,\eps}\}\to 1,\quad\eps\to 0,
\]
and thus it makes sense to study the asymptotic behavior of $t_\eps$ and $Y(t_\eps)$. 

To state our main result on system~\eqref{eq:main_conjugated_system},\eqref{eq:xi_0}, see Lemma~\ref{lm:main_linear_lemma} below, 
we have to introduce a certain multidimensional distribution
that is easier to describe in terms of random variables defined on some sufficiently
rich probability space. 

We start with the random vector $\xi_0$ introduced after~\eqref{eq:xi_0}.
Then, on the same probability space 
we define a $(\nu-1)$-dimensional centered Gaussian vector $(N^1_0,\ldots,N^{\nu-1}_0)$
with covariance
\begin{equation}
\E N^k_0 N^j_0 =\int_0^\infty e^{-(\lambda_k+\lambda_j) s}(BB^*)^{kj}(S^s_Ay_0)ds,\quad k,j<\nu,
\label{eq:Gaussian_along_the_stable}
\end{equation}
independent of $\xi_0$. Next, we define $(\kap^1,\ldots,\kap^{\nu-1})$ by
\begin{equation}
\kap^k=\xi^k_0+N_0^k\ONE_{\alpha=1},\quad k<\nu.
\label{eq:kappa}
\end{equation}

Finally, on the same probability space we define a random vector $(\bar N^\nu_0,\ldots,\bar N^{d}_0)$.
Conditioned on each of the two events $\{\sgn(\kap^1)=\pm1\}$, 
 it is a centered Gaussian vector with
\[
\E \bar N^k_0 \bar N^j_0 = \int_{-\infty}^0 e^{-(\lambda_k+\lambda_j)s} (BB^*)^{kj}(S_A^s(\pm Rv_1))ds, \quad k,j\ge\nu,
\]
and independent of $(|\kap^1|,\kap^2,\ldots,\kap^{\nu-1})$.

\begin{lemma}
\label{lm:main_linear_lemma} 
 Suppose the following nondegeneracy conditions are satisfied:
\[y_0^\nu\ne0,\]
\[
\Pp\{\xi^1_0\ne0\}=1,\quad\text{if}\ \alpha<1,
\]
\[
\Pp\{\xi^2_0\ne0\}=1,\quad\text{if}\ \alpha<1\ \text{and}\ \nu>2.
\]
Then
there exists a number $\beta$, and a random vector $(y', \xi', \zeta)$ such that the random vector
\begin{equation*}
\left(Y^1_\eps(t_\eps),\ \eps^{-\beta}\Pi_L Y_\eps(t_\eps),\ t_\eps-\frac{\alpha}{\lambda_1}\ln\frac{1}{\eps}\right)
\end{equation*}
converges in distribution  to $(y', \xi', \zeta)$. More precisely,
\begin{equation*}
\beta=
\begin{cases}
1,& \nu=2\ \text{\rm and}\  -\alpha\lambda_\nu\ge \lambda_1,\\
-\alpha\frac{\lambda_\nu}{\lambda_1},& \nu=2\ \text{\rm and}\ -\alpha\lambda_\nu<\lambda_1,\\
\alpha\left(1-\frac{\lambda_1}{\lambda_2}\right),& \nu>2\ \text{\rm and}\ -\lambda_\nu\ge \lambda_1-\lambda_2,\\
-\alpha\frac{\lambda_\nu}{\lambda_1},&\nu>2\ \text{\rm and}\ -\lambda_\nu< \lambda_1-\lambda_2,
\end{cases}
\end{equation*}

\medskip

\begin{equation*}
y'=\sgn(\kap^1)R,
\end{equation*}
and
\begin{equation}
\label{eq:exit_distribution}
\xi'=
\begin{cases}
N,& \nu=2\ \text{\rm and}\  -\alpha\lambda_\nu>\lambda_1,\\
 N + \eta_-,& \nu=2\ \text{\rm and}\  -\alpha\lambda_\nu=\lambda_1,\\
 \eta_-,& \nu=2\ \text{\rm and}\ -\alpha\lambda_\nu<\lambda_1,\\
\eta_+,& \nu>2\ \text{\rm and}\ -\lambda_\nu> \lambda_1-\lambda_2,\\
\eta_+ + \eta_-,& \nu>2\ \text{\rm and}\ -\lambda_\nu= \lambda_1-\lambda_2,\\
\eta_-,&\nu>2\ \text{\rm and}\ -\lambda_\nu< \lambda_1-\lambda_2,
\end{cases}
\end{equation}
with
\[
\eta_-=R^{\frac{\lambda_\nu}{\lambda_1}}|\kap^1|^{-\frac{\lambda_\nu}{\lambda_1}}y_0^\nu v_\nu,
\]
\[
\eta_+=R^{\frac{\lambda_2}{\lambda_1}}|\kap^1|^{-\frac{\lambda_2}{\lambda_1}}\kap^2 v_2.
\]
and
$$
N=\sum_{k=\nu}^d \bar N^k_0 v_k.
$$
\end{lemma}
\begin{remark}\rm Even if the nondegeneracy assumptions do not hold, a version of this lemma
still holds true. This will be obvious from the proof, and, for brevity, we omit a variety of related results on these degenerate 
situations.  
\end{remark}
\begin{remark} \label{rem:non-markov}\rm 
We see that of all random variables $N$, $\eta_-$, $\eta_+$,
involved in the description of the limit, conditioned on $\sgn(\kap^1)$, only $N$ does not depend
in any way on the initial distribution data given by $y_0$, $\alpha$,
and $\xi_0$. This guarantees
the Markovian loss of memory for the case $[\nu=2; -\alpha\lambda_\nu>\lambda_1],$ and potentially leads to non-Markov
situations in all the other cases, see Section~\ref{sec:discussion} for further discussion.
\end{remark}

The proof consists of two parts. The first part provides the analysis of the evolution of $Y_\eps$ mostly along the
stable manifold. The second part is mostly responsible for the motion along the unstable manifold of the origin. 

Using It\^o's formula it is easy to verify that Duhamel's principle holds:
$$
Y_\eps(t)=e^{At}Y_\eps(0)+\eps e^{At}\int_0^{t}e^{-As}B(Y_\eps(s))dW(s) +\eps^2e^{At}\int_0^{t}e^{-As}C(Y_\eps(s))ds,
$$
or, equivalently,
\begin{equation}
Y^{k}_\eps(t)=e^{\lambda_kt}\left(y_0^k+\eps^\alpha\xi_\eps^k
+\eps\int_0^{t}e^{-\lambda_k s}B^k(Y_\eps(s))dW(s)+\eps^2\int_0^{t}e^{-\lambda_ks}C^k(Y_\eps(s))ds\right).
\label{eq:Duhamel_for_Y_in_coordinates}
\end{equation}

We start with a study of the outcome of the evolution of $Y_\eps$ along the stable manifold. 
We fix a number $\bar\alpha\in(0,\alpha)$.
Our first goal is to analyze the distribution of $Y_\eps(\tau_\eps)$, where
\[
\tau_\eps=\min\{\tau_\eps^k,\ k=1,\ldots,d\},
\]
and for every $k=1,\ldots,d$,
\[
\tau_\eps^k=\inf\{t:|Y^k_{\eps}(t)-(S^t_Ay_0)^k|=\eps^{\bar\alpha}\}.
\]


\begin{lemma}\label{lm:convergence_along_stable}
\[
\sup_{t\ge0}|Y_{\eps}(\tau_\eps\wedge t)-(S^{t\wedge\tau_\eps}_Ay_0)|\stackrel{\Pp}{\to}0,\quad \eps\to\infty.
\]
\end{lemma}
\begin{proof} This statement is obvious due to 
\begin{equation}
\label{eq:bounded_by_eps_alpha}
\sup_{t\ge0}|Y_{\eps}(\tau_\eps\wedge t)-(S^{t\wedge\tau_\eps}_Ay_0)|\le d\eps^{\alpha},
\end{equation}
which follows from the definition of $\tau_\eps$.
\end{proof}

\begin{lemma}\label{lm:tau-to-infty}
\[
\tau_\eps\stackrel{\Pp}{\to}\infty,\quad \eps\to\infty.
\]
\end{lemma}

A proof based on \eqref{eq:Duhamel_for_Y_in_coordinates} is given in Section~\ref{sec:proof_tau_to_infty}.

\begin{lemma}\label{lm:tau=tau_1}
\[
\Pp\{\tau_\eps=\tau_\eps^1\}{\to}1,\quad \eps\to\infty.
\]
\end{lemma}
A sketch of a proof is also  given in Section~\ref{sec:proof_tau_to_infty}.

Let us now take a closer look at the stochastic integral term in the expression~\eqref{eq:Duhamel_for_Y_in_coordinates}.
For $k<\nu$ we introduce
\begin{equation}
\label{eq:Nk}
N^k_\eps(t)=\int_0^t e^{-\lambda_k s}B^k(Y_\eps(s))dW(s)=\sum_{m=1}^d\int_0^t e^{-\lambda_k s}B^k_m(Y_\eps(s))dW^m(s)
\end{equation}
and
\[
M^k(t)=\int_0^t e^{-\lambda_k s}B^k(S^s_Ay_0)dW(s)=\sum_{m=1}^d\int_0^t e^{-\lambda_k s}B^k_m(S^s_Ay_0)dW^m(s).
\]
A straightforward application based on BDG inequalities (see e.g.~\cite[Theorem 3.28, Chapter 3]{Karatzas--Shreve}),
local Lipschitzness of $B$, and \eqref{eq:bounded_by_eps_alpha} implies that, as $\eps\to 0$, 
\[
\sup_{t\le\tau_\eps}\left|N^k_\eps(t)-M^k(t)\right|=\sup_{t\le\tau_\eps}
\left|\int_0^t e^{-\lambda_k s}(B^k(Y_\eps(s))-B^k(S^s_Ay_0))dW(s)\right|\stackrel{\Pp}{\to}0.
\]
Lemma~\ref{lm:tau-to-infty}
implies that
the terminal value $N_\eps(\tau_\eps)$
converges to  $M(\infty)$. 
Computing $\E M^k(\infty) M^j(\infty)$, we see that $\Law(M(\infty))=\Law(N_0)$,
where~$N_0$ is a centered Gaussian  vector defined in \eqref{eq:Gaussian_along_the_stable}. Therefore,
\[
N_\eps(\tau_\eps)\stackrel{\Law}{\to} N_0,\quad\eps\to0.
\]
We also notice that
\[
\sup_{t\le\tau_\eps}\left|\eps^2\int_0^{t}e^{-\lambda_ks}C^k(Y_\eps(s))ds\right|=o_\Pp(\eps),
\]
where we use the notation $\phi(\eps)=o_\Pp(\psi(\eps))$ for any families of random variables $\phi(\eps),\psi(\eps),\eps>0$ such that
$\frac{\phi(\eps)}{\psi(\eps)}\stackrel{\Pp}{\to}0$ as $\eps\to0$.

Therefore, for $k<\nu$ we have
\begin{equation}
Y^{k}_\eps(\tau_\eps)=e^{\lambda_k\tau_\eps}(\eps^\alpha\xi_\eps^k+\eps N^k_\eps(\tau_\eps)+o_\Pp(\eps)).
\label{eq:Y_at_tau}
\end{equation}
We denote
\begin{equation*}
\kap^k_\eps = Y^{k}_\eps(\tau_\eps)e^{-\lambda_k\tau_\eps}\eps^{-\alpha},
\end{equation*}
so that
\[
\kap^k_\eps=\xi_\eps^k+\eps^{1-\alpha} N^k_\eps(\tau_\eps)+o_\Pp(\eps^{1-\alpha}),
\]
and rewrite~\eqref{eq:Y_at_tau} as
\begin{equation}
Y^{k}_\eps(\tau_\eps)=e^{\lambda_k\tau_\eps}\eps^\alpha\kap^k_\eps.
\label{eq:Y_at_tau_via_kappa}
\end{equation}

Lemma~\ref{lm:tau=tau_1} implies that 
\begin{equation}
\label{eq:Y1(tau_eps)}
\Pp\{|Y^{1}_\eps(\tau_\eps)|=\eps^{\bar\alpha}\}\to1.
\end{equation}
 So, taking logarithms of
\begin{equation*}
e^{\lambda_1\tau_\eps}\eps^\alpha|\kap^1_\eps|=\eps^{\bar\alpha}
\end{equation*}
we see that with probability approaching $1$, as $\eps\to0$, 
\begin{equation}
\label{eq:tau-asymptotics}
\tau_\eps-\frac{\bar\alpha-\alpha}{\lambda_1}\ln\eps=-\frac{\ln|\kap^1_\eps|}{\lambda_1} 
\end{equation}
and
\begin{equation}
\label{eq:sign_of_Y}
Y_\eps^1(\tau_\eps)\eps^{-\bar\alpha}= \sgn(\kap^1_\eps),
\end{equation}

Plugging~\eqref{eq:tau-asymptotics} into~\eqref{eq:Y_at_tau_via_kappa}, we obtain the following
result:
\begin{lemma} For $1<k<\nu$,
\label{lm:positive_exponents_at_exit}
\begin{equation*}
Y^{k}_\eps(\tau_\eps)=
\eps^{\frac{\lambda_k}{\lambda_1}(\bar\alpha-\alpha)+\alpha}
\kap^k_\eps|\kap^1_\eps|^{-\frac{\lambda_k}{\lambda_1}}.
\end{equation*}
\end{lemma}
This lemma describes the asymptotics of $Y^{k}_\eps(\tau_\eps),k<\nu$ very
precisely since
due to~\eqref{eq:tau-asymptotics} we have the following obvious statement:
\begin{lemma}\label{lm:convergence_of_kappa}
\[
\left(\xi_\eps,\tau_\eps-\frac{\bar\alpha-\alpha}{\lambda_1}\ln\eps,
\kap^1_\eps,\ldots,\kap^{\nu-1}_\eps\right)
\stackrel{\Law}{\to}
\left(\xi_0,-\frac{\ln|\kap^1|}{\lambda_1},
\kap^1,\ldots,\kap^{\nu-1}\right),
\]
where the random variables in the r.h.s. were defined before the statement of  Lemma~\ref{lm:main_linear_lemma}.
\end{lemma}


We shall now consider $k\ge\nu$. For these values of $k$, we denote 
\[
R^k_\eps(t)=e^{\lambda_kt}\int_0^t e^{-\lambda_ks}B^k(Y_\eps(s))dW(s).
\]
Let us take any numbers $T\in\N$ and $p>0$, and use BDG inequalities to write
\begin{align}\notag
\Pp\left\{\sup_{s\in[0,T]}|R^k_\eps(t\wedge\tau_\eps)|>\eps^{-p}\right\}&\le\sum_{n=0}^{T-1}
\Pp\left\{\sup_{t\in[n,n+1]}\left|R^k_\eps(t\wedge\tau_\eps)\right|>\eps^{-p}\right\}\\ \notag
&\le\sum_{n=0}^{T-1}\Pp\left\{\sup_{t\le n+1}\left|\int_0^{t\wedge\tau_\eps} e^{-\lambda_ks}B^k(Y_\eps(s))dW(s)\right|\ge e^{-\lambda_kn}\eps^{-p}\right\}
\\&\le  KT\eps^{-2p},
\label{eq:o(eps-2p)}
\end{align}
for some $K$.

Since the term
\[
e^{\lambda_kt}\int_0^t e^{-\lambda_ks}C^k(Y_\eps(s))ds
\]
is bounded on $t\le\tau_\eps$, we see that \eqref{eq:Duhamel_for_Y_in_coordinates},\eqref{eq:tau-asymptotics} and \eqref{eq:o(eps-2p)}
imply the following result:
\begin{lemma}\label{lm:negative_exponents_at_exit}
 For all $k\ge\nu$ and any $p>0$,
\begin{equation*}
Y_\eps^k(\tau_\eps)=\eps^{\frac{\lambda_k}{\lambda_1}(\bar\alpha-\alpha)}|\kap^1_\eps|^{-\frac{\lambda_k}{\lambda_1}}(y_0^k+\eps^\alpha\xi_\eps^k)+o_\Pp(\eps^{1-p}).
\end{equation*}
\end{lemma}

Lemmas \ref{lm:positive_exponents_at_exit}, \ref{lm:convergence_of_kappa}, and \ref{lm:negative_exponents_at_exit}
provide all the necessary information on the behavior of $Y_\eps$ up to $\tau_\eps$. We now turn to the second part of the proof, the
analysis
of the evolution of $Y_\eps$ after $\tau_\eps$.

Let $\bar Y_\eps(t)=Y_\eps(\tau_\eps+t)$.
We shall study $\bar Y_\eps$ conditioned on 
$\bar Y^1(0)=Y_\eps^1(\tau_\eps)=\pm\eps^{\bar\alpha}$. We consider only the case of $+\eps^{\bar\alpha}$, 
since the analysis of the other case is entirely the same.
First, we rewrite \eqref{eq:Duhamel_for_Y_in_coordinates} for $\bar Y_\eps$:
\begin{equation}
\bar Y^{k}_\eps(t)=e^{\lambda_kt}\left(Y^k_\eps(\tau_\eps)
+\eps\int_0^{t}e^{-\lambda_k s}B^k(\bar Y_\eps(s))dW(s)+\eps^2\int_0^{t}e^{-\lambda_ks}C^k(\bar Y_\eps(s))ds\right).
\label{eq:Duhamel_for_bar_Y_in_coordinates}
\end{equation}
Let $\bar\tau_\eps=\inf\{t:\ |\bar Y_\eps^1|=R\}.$
We can rewrite~\eqref{eq:Duhamel_for_bar_Y_in_coordinates} for $k=1$, $t=\bar\tau_\eps$ as
\begin{equation}
\bar Y^{1}_\eps(\bar\tau_\eps)=e^{\lambda_1\bar\tau_\eps}\eps^{\bar\alpha}(1+\eta_\eps),
\label{eq:Duhamel_for_bar_Y_1}
\end{equation}
with
\[
\eta_\eps =\eps^{1-\bar\alpha}\int_0^{\bar\tau_\eps}e^{-\lambda_1 s}B^1(\bar Y_\eps(s))dW(s)+\eps^{2-\alpha}\int_0^{\bar\tau_\eps}e^{-\lambda_1s}C^1(\bar Y_\eps(s))ds\stackrel{\Pp}{\to} 0.
\]
The last relation is obvious if $B$ and $C$ are bounded. In the general case it follows from a localization argument.

Relation~\eqref{eq:Duhamel_for_bar_Y_1} implies
\begin{equation}
\bar\tau_\eps=-\frac{\bar\alpha}{\lambda_1}\ln\eps+\frac{1}{\lambda_1}\ln\frac{R}{1+\eta_\eps}.
\label{eq:asymptotics_for_exit_at_1} 
\end{equation}
Plugging this into~\eqref{eq:Duhamel_for_bar_Y_in_coordinates} and applying  
Lemmas \ref{lm:positive_exponents_at_exit} and \ref{lm:negative_exponents_at_exit} we can prove
the following statement:
\begin{lemma}
\label{lm:uniform_convergence_out}
\begin{equation}
\label{eq:uniform_convergence_out}
\sup_{t\le\bar\tau_\eps}|\bar Y_\eps(t)-S^t_A (\eps^{\bar\alpha}v_1)|\stackrel{\Pp}{\to} 0.
\end{equation}
\end{lemma}




However, we need a more detailed information on $\bar Y_\eps(\bar \tau_\eps)$. To that end,
we analyze $\bar Y_\eps^k(\bar \tau_\eps)$ separately for $2\le k<\nu$ and $\nu\le k\le d$.

For $2\le k<\nu$, plugging~\eqref{eq:asymptotics_for_exit_at_1} 
into~\eqref{eq:Duhamel_for_bar_Y_in_coordinates}, using Lemma~\ref{lm:positive_exponents_at_exit}
for the first term and
Lemma~\ref{lm:uniform_convergence_out} to estimate the integral terms, we see that 

\begin{align}\notag
\bar Y_\eps^k(\bar\tau_\eps)&=\eps^{-\frac{\lambda_k}{\lambda_1}\bar\alpha}
R^{\frac{\lambda_k}{\lambda_1}}(1+\eta_\eps)^{-\frac{\lambda_k}{\lambda_1}}\left(\eps^{\frac{\lambda_k}{\lambda_1}(\bar\alpha-\alpha)+\alpha}
\kap_\eps^k|\kap_\eps^1|^{-\frac{\lambda_k}{\lambda_1}}+o_\Pp(\eps^{\frac{\lambda_k}{\lambda_1}(\bar\alpha-\alpha)+\alpha})\right)\\
&=\eps^{\alpha(1-\frac{\lambda_k}{\lambda_1})}R^{\frac{\lambda_k}{\lambda_1}}
\kap_\eps^k|\kap_\eps^1|^{-\frac{\lambda_k}{\lambda_1}}(1+o_\Pp(1)),\quad 2\le k<\nu.
\label{eq:exit_for_positive_lambda}
\end{align}


For $k\ge\nu$ we denote 
\[\bar N^k_\eps=e^{\lambda_k \bar\tau_\eps}\int_0^{\bar\tau_\eps}e^{-\lambda_k s}B^k(Y_\eps(s))dW(s)\]
\begin{lemma}\label{lm:convergence_to_Gaussian_at_exit_negative_lambdas} As $\eps\to0$,
\[
(\bar N^\nu_\eps,\ldots,\bar N^{d}_\eps)\stackrel{\Law}{\to}(\bar N^\nu_0,\ldots,\bar N^{d}_0),
\]
where $(\bar N^\nu_0,\ldots,\bar N^{d}_0)$ is the centered Gaussian vector defined before the statement
of Lemma~\ref{lm:main_linear_lemma} 

\end{lemma}
\begin{proof}

Denoting
\[\tau'_\eps=-\frac{\bar\alpha}{\lambda_1}\ln\eps+\frac{1}{\lambda_1}\ln R\to\infty,\]
we obtain
\[\tau'_\eps\to\infty,
\quad\eps\to0\]
and
\begin{equation*}
\bar\tau_\eps-\tau'_\eps \stackrel{\Pp}{\to}0,\quad\eps\to0.
\end{equation*}

Therefore,
\begin{equation}
\label{eq:approximation_for_bar_N}
\bar N^k_\eps-e^{\lambda_k \tau'_\eps }\int_0^{\tau'_\eps}e^{-\lambda_k s}B^k(Y_\eps(s))dW(s)\stackrel{\Pp}{\to}0.
\end{equation}
It follows from Lemma~\ref{lm:uniform_convergence_out} that
\begin{equation}
e^{\lambda_k \tau'_\eps }\int_0^{\tau'_\eps}e^{-\lambda_k s}B^k(Y_\eps(s))dW(s)
-\hat N^k_\eps\stackrel{\Pp}{\to}0,
\label{eq:closeness_inprobability}
\end{equation}
where
\[
\hat N^k_\eps=e^{\lambda_k \tau'_\eps }\int_0^{\tau'_\eps}e^{-\lambda_k s}B^k(S^{s}_A\eps^{\bar \alpha}v_1)dW(s),\quad k\ge\nu.
\]
We see that $(\hat N^\nu_\eps,\ldots, \hat N^d_\eps)$
is a centered Gaussian vector with
\begin{align}\notag
\E \hat N^j_\eps \hat N^k_\eps& = \int_0^{\tau'_\eps}e^{(\lambda_k+\lambda_j) (\tau'_\eps-s)}(BB^*)^{kj}(S^{s}_A\eps^{\bar \alpha}v_1)ds\\
&=\int_{-\tau'_\eps}^0 e^{-(\lambda_k+\lambda_j)r}(BB^*)^{kj}(S^{\tau'_\eps+r}_A\eps^{\bar \alpha}v_1)dr \notag \\
&\to \int_{-\infty}^0 e^{-(\lambda_k+\lambda_j)r}(BB^*)^{kj}(S^{r}_ARv_1)dr,\quad\eps\to0,
\label{eq:convergence_of_covariance}
\end{align}
where the second indentity follows from the change of variables $s-\tau'_\eps=r$, and the convergence in the last
line is implied by the uniform convergence $S^{\tau'_\eps+r}_A\eps^{\bar \alpha}v_1\to S^{r}_ARv_1$, $r\le0$.

Lemma~\ref{lm:convergence_to_Gaussian_at_exit_negative_lambdas}
follows now from \eqref{eq:approximation_for_bar_N}--\eqref{eq:convergence_of_covariance}. 
\end{proof}

Now, for $k\ge\nu$, equations \eqref{eq:Duhamel_for_bar_Y_in_coordinates},\eqref{eq:asymptotics_for_exit_at_1}, and Lemma~\ref{lm:negative_exponents_at_exit} imply:
\begin{align}\notag
\bar Y_\eps^k(\tau_\eps)&=
e^{\lambda_k\left(-\frac{\bar\alpha}{\lambda_1}\ln\eps+\frac{1}{\lambda_1}\ln\frac{R}{1+\eta_\eps}\right)}
\left(\eps^{\frac{\lambda_k}{\lambda_1}(\bar\alpha-\alpha)}|\kap^1_\eps|^{-\frac{\lambda_k}{\lambda_1}}(y^k_0+\eps^\alpha\xi_\eps^k)+o_\Pp(\eps^{1-p})\right)
+\eps\bar N_\eps^k+o_\Pp(\eps)
\\ &=\eps^{-\frac{\lambda_k}{\lambda_1}\alpha} R^{\frac{\lambda_k}{\lambda_1}} |\kap^1_\eps|^{-\frac{\lambda_k}{\lambda_1}}y^k_0(1+o_\Pp(1))
+o_\Pp(\eps^{1-p-\frac{\lambda_k}{\lambda_1}\bar\alpha})+\eps\bar N_\eps^k+o_\Pp(\eps).
\label{eq:exit_for_negative_lambda}
\end{align}

\bigskip

We are now ready to finish the proof of Lemma~\ref{lm:main_linear_lemma}. We notice
that formulas analogous to~\eqref{eq:exit_for_positive_lambda} and \eqref{eq:exit_for_negative_lambda} 
hold true if we condition on $\bar Y^1_\eps(0)=Y^1_\eps(\tau_\eps)=-\eps^\alpha$. 
Since $\bar Y_\eps(\bar \tau_\eps)=Y_\eps(t_\eps)$ and $t_\eps=\tau_\eps+\bar\tau_\eps$,
Lemma~\ref{lm:main_linear_lemma} is a consequence of the strong Markov property and
an elementary analysis of relations \eqref{eq:tau-asymptotics},\eqref{eq:asymptotics_for_exit_at_1},
\eqref{eq:exit_for_positive_lambda}, \eqref{eq:exit_for_negative_lambda}.
The proof reduces to  
extracting the leading order terms in \eqref{eq:exit_for_positive_lambda}, \eqref{eq:exit_for_negative_lambda}
in each of the cases that appear in the statement of the lemma.

For example, in the case of $\nu>2$ and $-\lambda_\nu< \lambda_1-\lambda_2$,
the greatest contribution (as $\eps\to0$) in \eqref{eq:exit_for_positive_lambda}
is of order of $\eps^{\alpha(1-\frac{\lambda_2}{\lambda_1})}$ and corresponds to $k=2$.
The greatest contribution in \eqref{eq:exit_for_negative_lambda} is of order of
$\eps^{-\frac{\lambda_\nu}{\lambda_1}\alpha}$ and corresponds to $k=\nu$ since one can choose
$p$  sufficiently small and neglect the asymptotic
 contribution of the $o_\Pp(\eps^{1-p-\frac{\lambda_k}{\lambda_1}\bar\alpha})$ term. (Notice also that in this case
 the $o_\Pp(\eps)$ term can also be neglected since $\eps=o(\eps^{-\frac{\lambda_\nu}{\lambda_1}\alpha})$
 due to $\alpha\le1$ and $-\frac{\lambda_\nu}{\lambda_1}<\frac{\lambda_1-\lambda_2}{\lambda_1}<1$.)
Among these two contributions,  $\eps^{-\frac{\lambda_\nu}{\lambda_1}\alpha}$ dominates
providing the desired asymptotics
\[
\Pi_L Y_\eps(t_\eps)\sim \eps^{-\frac{\lambda_\nu}{\lambda_1}\alpha} R^{\frac{\lambda_\nu}{\lambda_1}} |\kap^1_\eps|^{-\frac{\lambda_k\nu}{\lambda_1}}y^\nu_0 v_\nu.
\]

The analysis of all the other cases is similar, and we omit it.

\bigskip

Having Lemma~\ref{lm:main_linear_lemma} at hand, it is straightforward to write down the asymptotics for the original process 
$X_\eps(t_\eps)=g(Y_\eps(t_\eps))$.

\begin{lemma}
\label{lm:main_linear_lemma_back_to_nonlinear} Let $X_\eps$ solve the system~\eqref{eq:sde} with initial
condition~\eqref{eq:inidata_before_splitting} and assume $(x,\alpha,\mu)\in \In_i$.  
Let us define $\kap^1,\kap^2$ via~\eqref{eq:xi_0},\eqref{eq:Gaussian_along_the_stable},\eqref{eq:kappa}, and 
$\beta,\xi',\zeta$ as in Lemma~\ref{lm:main_linear_lemma}.

Then
\begin{equation}
X_\eps(t_\eps)=q_\eps+\eps^\beta\phi'_\eps,
\end{equation}
where
\[
 \Pp\{q_\eps=g(\pm R)\}=1,\quad \eps>0,
\]
and
\begin{equation*}
\left(q_\eps,\phi'_\eps,\ t_\eps-\frac{\alpha}{\lambda_1}\ln\frac{1}{\eps}\right)\stackrel{\Law}{\to}(q, \phi', \zeta),\quad\eps\to0.
\end{equation*}
Here
\[q=g(\sgn(\kap^1)R),\]
and
\[\phi'=(Df(q))^{-1}\xi'.\]
\end{lemma}

\begin{remark}\rm The case of $\alpha=1$ and $\phi_\eps=0$ is also covered by this lemma.
This situation corresponds to the deterministic initial condition for all $\eps>0$.
\end{remark}

\section{Asymptotics along the heteroclinic orbit}\label{sec:along_heteroclinic}

In this section we consider the equation~\eqref{eq:sde}, equipped with
the initial condition
\[
X_\eps(0)=x_0+\eps^\alpha\phi_\eps,
\] 
 on a finite time horizon (up to a
nonrandom time $T$). Here $\alpha\in(0,1]$, and~$(\phi_\eps)_{\eps>0}$ is a family of random
vectors satisfying
\[
\phi_\eps\stackrel{\Law}{\to}\phi_0,\quad \eps\to0,
\]
for some nondegenerate random vector $\phi_0$, independent of the Wiener process driving the  equation. 

The following statement is elementary.
\begin{lemma}\label{lm:X_to_S_in_probability} As $\eps\to0$,
\begin{equation*}
\sup_{t\in[0,T]}|X_\eps(t)-S^tx_0|\stackrel{\Pp}{\to}0.
\end{equation*}
\end{lemma}

We are going to give more precise asymptotics in the spirit of~\cite{Blagoveschenskii:MR0139204}. To that end we denote
\[
Y(t)=X(t)-S^tx_0,
\]
and write
\begin{align*}
Y_\eps(t)&=\eps^\alpha\phi_\eps+\int_0^t(b(X_\eps(s)-b(S^tx_0)))ds+\eps\int_0^t\sigma(X_\eps(s))dW(s)\\
&=\eps^\alpha\phi_\eps+\int_0^t Db(S^tx_0)Y_\eps(s)ds+\int_0^t Q(S^tx_0,Y_\eps(s))ds +\eps\int_0^t\sigma(X_\eps(s))dW(s),
\end{align*}
where
\begin{equation}
\label{eq:quadratic_correction}
|Q(x,y)|\le Cy^2,
\end{equation}
for some constant $C$ and all $y$ with $|y|<1$.
Treating this as a linear nonhomogeneous equation we apply the Duhamel principle to see that
\begin{equation}
\label{eq:error_of_linearization}
Y_\eps(t)=\eps^\alpha \Phi_{x_0}(t)\phi_\eps+\eps \Phi_{x_0}(t)\int_0^t \Phi^{-1}_{x_0}(s)\sigma(X_\eps(s))dW(s)+\Phi_{x_0}(t)\int_0^t\Phi^{-1}_{x_0}(s)Q(S^tx_0, Y_\eps(s))ds.
\end{equation}
We introduce a stopping time $\tau_\eps$ by
\[
\tau_\eps=\inf\{t:\ |Y_\eps(t)|\ge\eps^{\frac{2}{3}(1\wedge\alpha)}\}.
\]
On the event $\{\tau_\eps\le T\}$ we use \eqref{eq:error_of_linearization} to write
\begin{equation}
Y_\eps(t\wedge\tau_\eps)=I_1(t,\eps)+ I_2(t,\eps)+I_3(t,\eps),
\label{eq:decomposition_of_Y}
\end{equation}
with
\[
|I_1(t,\eps)|=|\eps^\alpha \Phi_{x_0}(t\wedge \tau_\eps)\phi_\eps|\le K_1\eps^\alpha|\phi_\eps|  
\]
for some constant $K_1$,
\[
|I_2(t,\eps)|=\eps  \left| \Phi_{x_0}(t\wedge \tau_\eps)\int_0^{t\wedge\tau_\eps}\Phi^{-1}_{x_0}(s)\sigma(X_\eps(s))dW(s)\right|=\eps |N_\eps(t)|,
\]
where $N_\eps$ is a tight (in $C$ with $\sup$-norm) family of continuous processes due to Lemma~\ref{lm:X_to_S_in_probability},
and, finally,
\[
|I_3(t,\eps)|\le K_3\eps^{\frac{4}{3}(1\wedge\alpha)}
\]
for some constant $K_3$ due to \eqref{eq:quadratic_correction}.
It now follows from~\eqref{eq:decomposition_of_Y} for $t=T$, that on $\{\tau_\eps\le T\}$
\[
\eps^{\frac{2}{3}(1\wedge\alpha)}\le K_1\eps^\alpha|\phi_\eps| + \eps |N_\eps(\tau_\eps)| + K_3\eps^{\frac{4}{3}(1\wedge\alpha)},
\]
which automatically implies 
\[
\Pp\{\tau_\eps\le T\}\to 0,\quad\eps\to0.
\]
On the complement, $\{\tau_\eps> T\}$, we have
\begin{equation}
Y_\eps(T)= I'_1(\eps)+ I'_2(\eps)+I'_3(\eps),
\label{eq:decomposition_of_Y_at_T}
\end{equation}
where
\begin{equation}
\label{eq:I_1_prime}
I'_1(\eps)=\eps^\alpha \Phi_{x_0}(T)\phi_\eps,
\end{equation}
\begin{equation}
\label{eq:I_2_prime}
I'_2(\eps)=\eps\Phi_{x_0}(T)\int_0^{T}\Phi^{-1}_{x_0}(s)\sigma(X_\eps(s))dW(s),
\end{equation}
\begin{equation}
\label{eq:I_3_prime}
|I'_3(\eps)|\le K_4 \eps^{\frac{4}{3}(1\wedge\alpha)}
\end{equation}

Now \eqref{eq:decomposition_of_Y_at_T} --- \eqref{eq:I_3_prime} imply the following result:
\begin{lemma}\label{lm:linearization_on_finite_interval} For every $\eps>0$
\[
X_\eps(T)=S^Tx_0+\eps^\alpha\bar\phi_\eps,
\]
where
\[
\bar \phi_\eps\stackrel{\Law}{\to}\bar\phi_0,\quad\eps\to0,
\]
with
\[
\bar\phi_0 = \Phi_{x_0}(T)\phi_0+\ONE_{\alpha=1}N,
\]
$N$ being a Gaussian vector:
\[
N=\Phi_{x_0}(T)\int_0^{T}\Phi^{-1}_{x_0}(s)\sigma(S^sx_0)dW(s).
\]
\end{lemma}
\begin{remark}\rm The lemma also holds true for the case where $\alpha=1$ and $\phi_\eps=0$.
This situation corresponds to the deterministic initial condition for all $\eps>0$.
\end{remark}

\bigskip 
We can now finish the proof of Lemma~\ref{lm:iteration_lemma}.
\begin{proof}[Proof of Lemma~\ref{lm:iteration_lemma}] Parts \ref{it:convergence_of_times} and \ref{it:convergence_of_splitting_probabilities} follow directly from 
Lemma~\ref{lm:main_linear_lemma_back_to_nonlinear}. Part~\ref{it:distribution_in_the_tangent_space}
follows from consecutive application of 
Lemmas~\ref{lm:main_linear_lemma_back_to_nonlinear} and~\ref{lm:linearization_on_finite_interval} together
with Strong Markov property and the nondegeneracy assumption~\eqref{eq:nondegeneracy_in_other_coordinates}.
The last two statements of the Lemma are obvious in view of the analysis above, and we omit their proofs.
\end{proof}
\begin{remark}\rm
Notice that not only we are able to prove that the entrance-exit map with desired properties exists,
but we also can describe this map explicitly, tracing the details from Lemmas~\ref{lm:main_linear_lemma}
and~\ref{lm:linearization_on_finite_interval}. For example,
\[
p_{i,\pm}(x,\alpha,\mu)=\Pp\{\sgn(\kap^1)=\pm1\},
\]
where $\kap^1$ is the random variable constructed in Lemma~\ref{lm:main_linear_lemma} applied to
the linearization about $z_i$.
\end{remark}

\section{Discussion}\label{sec:discussion}

In this section we informally comment on implications of our main result and its possible extensions.

\subsection{The structure of the limiting process.}\label{sec:limiting_process}

According to our main result, the limiting probability for the solution of the system to evolve along
a sequence of heteroclinic connections is  computed by a recursive procedure
described in \eqref{eq:associated_sequence}--\eqref{eq:product_of_conditional_probabilities}.

The key result for the analysis of the Markov property for the limiting process is Lemma~\ref{lm:main_linear_lemma}, which gives the ``exit distribution''
$\Law(\xi')$ via the ``entrance distribution'' $\Law(\xi_0)$.

Analyzing~\eqref{eq:exit_distribution}, we see that in the
case where for each $i\in\Cc$ we have
$\nu_i=2$, and $-\lambda_{i,2}>\lambda_{i,1}$ (i.e. contraction is stronger than expansion)
our  iteration scheme gives that at each saddle point the limiting process chooses each of the
outgoing heteroclinic connections with equal probabilities, thus defining a simple random walk on the
network viewed as a directed graph. 

However, in general, the  limiting distribution
on sequences of saddles does not necessarily define a Markov chain. 
In fact, if the exit distribution defined in~\eqref{eq:exit_distribution} involves $\eta_-$, then
it is asymmetric. Therefore, at the next saddle point,
the choice between the two heteroclinic orbits is asymmetric as well, and probabilities are generically not equal 
to~$1/2$. In fact they may equal $0$ or $1$.

Let us now look at $\eta_+$, the other random variable involved in the exit distribution. If the entrance distribution
is symmetric, then $\eta_+$ also has symmetric distribution. On the other hand, if the entrance distribution is asymmetric,
then $\eta_+$ is also (typically) asymmetric.

Notice that there are plenty of situations where the entrance distribution is strongly asymmetric
(concentrated on a semiline) and the probabilities to choose one of the two outgoing connections are~$0$
and~$1$, i.e. the choice is asymptotically deterministic (and dependent on the history of the process).

So, there are many possibilities, but roughly the limiting random walk on the saddles
may be described as follows. The system starts evolving in a Markov fashion (choosing the next saddle out of two possible ones
with probabilities $1/2$ independently of the history of the process)
until it meets a saddle point at which the exit distribution becomes asymmetric. After that, 
the choice of one of the two heteroclinic connections  is not Markov any more (being defined by the entrance distribution
whose asymmetry is in turn defined by the history of the process.)
Then, at each saddle point
the following three things may happen: (i) new asymmetry is brought in due to the presence of $\eta_-$ in the
exit asymptotics; (ii) the asymmetry present in the entrance distribution is transferred to the exit distribution
by $\eta_+$; (iii) $\nu=2$ and the contraction is strong enough to ensure that the exit distribution does not involve
$\eta_-$ or $\eta_+$ thus being symmetric. In the first two cases the system remembers its past encoded in the asymmetry
of $\eta_\pm$. However, in the last case, the system loses all the memory and goes back to the ``Markov mode'' (which is just a convenient name for this phase of the system's evolution; of course the system is not truly Markov since some information from the past
is encoded in the fact that the system is in the ``Markov mode'' presently.) 

In particular, if every entrance distribution involved is either symmetric or strongly asymmetric
(this excludes the ``rare'' cases of $[\nu=2;-\alpha\lambda_\nu=\lambda_1]$ and $[\nu>2; -\lambda_\nu= \lambda_1-\lambda_2]$), then
the two splitting probabilities are either both equal to $1/2$, or equal $0$ and $1$ respectively.

\begin{figure}
    \epsfig{file=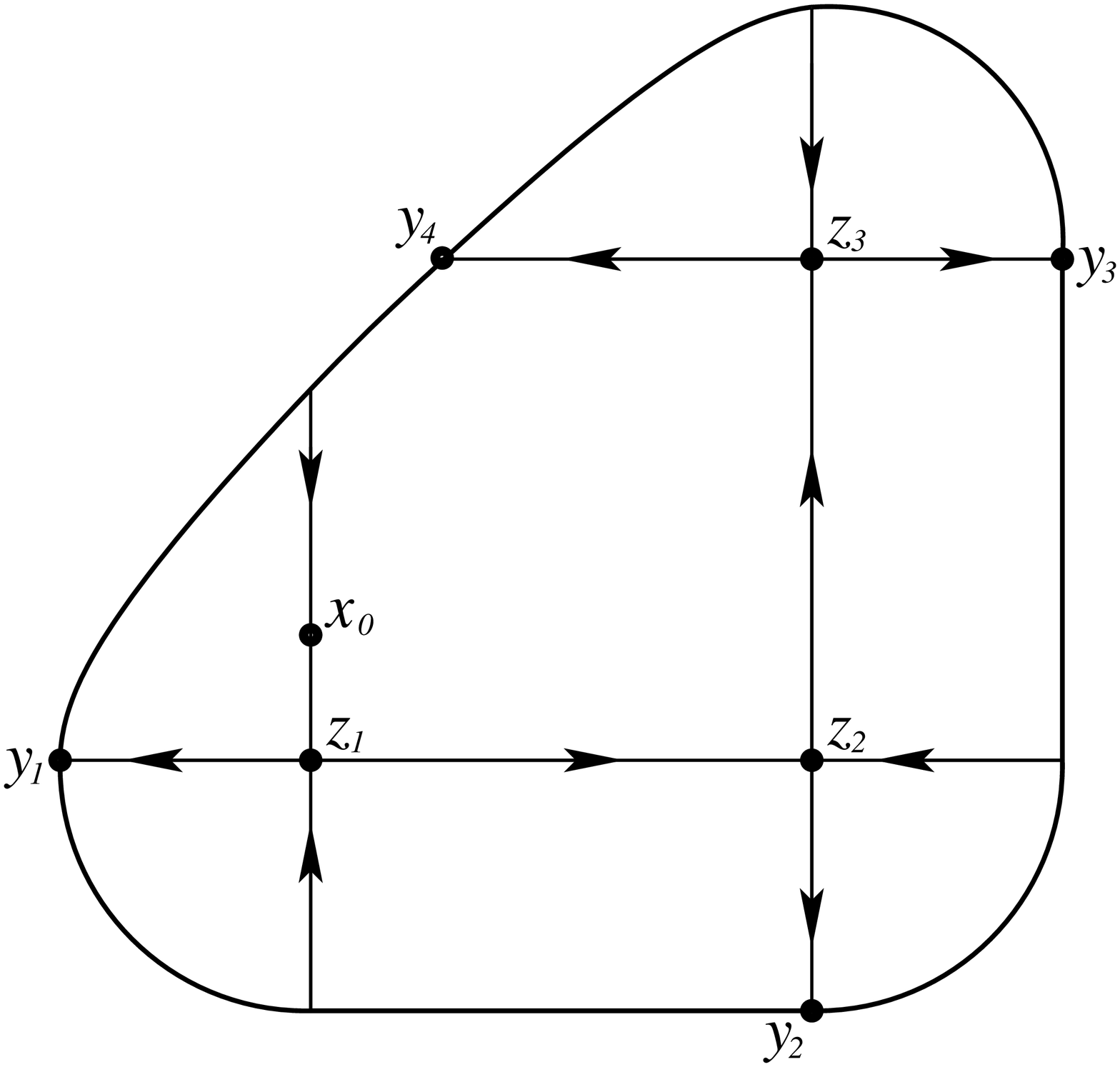,height=8cm}
\caption{} 
 \label{fig:cellular} 
\end{figure}

It is instructional to find the limiting exit measure for the 2-dimensional example shown on Figure~\ref{fig:cellular},
where the domain contains three saddle points $z_1,z_2$, and $z_3$. 
Let us assume that the linearizations at these points have the same eigenvalues $\lambda_1>0>\lambda_2$. 
According to Theorem~\ref{th:exit_asymptotics}, the exit measure will weakly converge, as $\eps\to 0$, to
\[p_1 \delta_{y_1}+p_2 \delta_{y_2}+p_3 \delta_{y_3}+p_4 \delta_{y_4},
\]
where $y_1,\ldots,y_4$ are the points on the boundary that are reachable along the network.

If $-\lambda_2>\lambda_1$ then
we have Markov evolution along the network resulting in
$p_1=1/2$, $p_2=1/4$, $p_3=p_4=1/8$. 

If $-\lambda_2<\lambda_1$, then the exit distribution at each saddle point is strongly asymmetric which
results in deterministic motion after the first bifurcation so that $p_1=1/2$, $p_2=0$, $p_3=0$, $p_4=1/2$.

If $-\lambda_2=\lambda_1$, then the exit distribution $N+\eta_-$ at each saddle point is asymmetric, so that the evolution is not Markov,
and $p_1=1/2$, $0<p_2<1/4$, $0<p_3<p_4$.

\subsection{Dependence of the exit measure on the diffusion coefficient.} Suppose that the vector field $b$
is given. From our basic iteration procedure we see that if for some matrix-valued function $\sigma$ at each saddle point the entrance distribution
is either symmetric or strongly asymmetric, then for any other choice of diffusion coefficients, the situation is the same, so that the
limiting process given by Theorem~\ref{th:main} and the exit measure asymptotics given by Theorem~\ref{th:exit_asymptotics} do not depend
of $\sigma$. Therefore, these theorems describe an intrinsic property of the vector field $b$.

\subsection{Our set of assumptions and possible generalizations.}
\label{sbsec:assumption_discussion}
We begin with an obvious and important extension of our theorem that we are going to mention without any proof.
We assumed so far that every critical point in the network is a saddle with nontrivial stable and unstable manifolds.
Let us see what happens if we allow situations where some of the critical points of the network are actually attracting nodes (this corresponds to $\nu_i=1$,
where $\nu_i$ is the order number of the first negative eigenvalue of the linearization at $z_i$). It follows from the FW theory that if the initial
condition belongs to a neighborhood of such a point, then with probability approaching 1 as $\eps\to0$, the trajectory will not leave that
neighborhood within the time of order $\ln(\eps^{-1})$. Therefore, our main result holds true for this situation as well, with
the following correction: with probability 1, for every $i\in\Cc$ with $\nu_i=1$, if the limiting process reaches~$z_i$, then it stays there indefinitely.

Another essential assumption that we have used is that near each saddle point the dynamics is $C^2$-smoothly 
conjugated to a linear flow.
Although this assumption leaves aside a number of interesting systems (one example is
2-dimensional Hamiltonian
dynamics), it is not too much
restrictive. Namely,
a smooth conjugation to linear dynamics is possible for saddle points satisfying so called no resonance conditions, see
a discussion of local conjugation to a linear flow in \cite[Section 6.6]{Katok--Hasselblatt}). In most examples of
heteroclinic networks in the survey~\cite{Krupa:MR1437986} all the saddles satisfy the no
resonance conditions for almost all parameters involved in the description of the system.

Although our proof relies entirely on the analysis of the conjugated linear systems, we believe that our results can be
extended to the situation with local dynamics that has no smooth linearization. The most promising way to approach the
understanding of that situation is to use normal forms (see \cite[Section 6.6]{Katok--Hasselblatt}).

Our assumption that all the eigenvalues of matrices $A_i$ are real and simple can be weakened. We
could have actually assumed that $\lambda_1,\lambda_2,\lambda_\nu\in\R$ are simple, and
\[
\lambda_1>\lambda_2>\Re\lambda_3>\ldots>0>\lambda_{\nu}>\ldots>\Re\lambda_{\nu+1}>\ldots.
\]

Obviously, the case where the eigenvalue with greatest real part has a nonzero imaginary part leads to a completely different picture since the invariant manifold
associated to $\lambda_1$ is two-dimensional in that case. The case where $\lambda_2$ or $\lambda_{\nu}$ are not real has a lot in common
with the case considered in this paper, since for small values of $\eps$ the system tends to evolve along heteroclinic connections. However,
instead of a unique limit distribution for~$Z_\eps$, there may be a set of limit distributions, i.e., the attractor in the space of
distributions may be non-trivial.
The reason is that due to the rotation associated with the imaginary part, the exit distribution at a saddle point is not well-defined. An analogue of Lemma~\ref{lm:main_linear_lemma}  holds true, but the normalization by $\eps^{-\beta}$ should be replaced by a more sophisticated one
involving a rotation by an angle proportional to $\ln(\eps^{-1})$.
Even more complicated situations arise if, instead of connecting hyperbolic fixed points, the heteroclinic orbits connect 
limit cycles or other invariant sets.

The last assumption we would like to discuss is the assumption that the set $L$ is conservative. Notice that
Theorem~\ref{th:main} does not necessarily describe the asymptotic behavior of the diffusion up to infinite
time. Indeed, times spent at subsequently visited saddles can become smaller and smaller so that the total time
accumulates to a finite value. There is no unique
deterministic time horizon that serves all limiting trajectories at once, since these values may differ from sequence
to sequence. If one chooses $L$ to be conservative, then one can find a valid trajectory-dependent (therefore, random) time
horizon serving all sequences from $L$.

\section{Auxiliary lemmas}
\label{sec:proof_tau_to_infty}

\begin{proof}[Proof of Lemma~\ref{lm:tau-to-infty}.]
We prove that $\tau_\eps^k\stackrel{\Pp}{\to}\infty$ for
every $k$. Due to \eqref{eq:Duhamel_for_Y_in_coordinates}, for any $T>0$,
\begin{equation}\label{eq:decomp-A1-A2-A3}
\Pp\{\tau_\eps^k\le T\}\le \Pp(A_1(\eps,T))+\Pp(A_2(\eps,T))+\Pp(A_3(\eps,T)),
\end{equation}
where
\begin{align*}
A_1(\eps,T)& =\left\{\eps^\alpha (1\vee e^{\lambda_k T})|\xi_\eps^k|>
\frac{\eps^{\bar\alpha}}{4}\right\},\\
A_2(\eps,T)& =\left\{\tau_\eps\le T;\ \exists t<\tau_\eps,\ : \eps e^{\lambda_k t}
\left|\int_0^t e^{-\lambda_k s}B^k(Y_\eps(s))dW(s)\right|>
\frac{\eps^{\bar\alpha}}{4}\right\},\\
A_3(\eps,T)& =\left\{\tau_\eps\le T;\ \exists t<\tau_\eps: \eps^2 e^{\lambda_k t}
\left|\int_0^t e^{-\lambda_k s}C^k(Y_\eps(s))ds\right|>
\frac{\eps^{\bar\alpha}}{4}\right\}.
\end{align*}

Now
\[
\Pp(A_1(\eps,T))\le\Pp\left\{|\xi_\eps^k|>\frac{\eps^{\bar\alpha-\alpha}}{4(1\vee e^{\lambda_kT})}\right\}\to 0,\quad\eps\to0,
\]
since $\eps^{\bar\alpha-\alpha}\to\infty$.

To estimate $A_2(\eps,T)$, we can write 
\begin{equation*}
\Pp(A_2(\eps,T))\le\Pp\left\{\sup_{t\le \tau_\eps\wedge T}|N^{k}_\eps(t)|>
\frac{\eps^{\bar\alpha-1}}{4(1\vee e^{\lambda_kT})}\right\}
=O(\eps^{2(1-\bar\alpha)})\to0,
\end{equation*}
which follows from consecutive application of the Chebyshev inequality, BDG inequalities, and an elementary estimate on the quadratic variation
of $N^{k}_\eps$ which was defined in~\eqref{eq:Nk}.

Next, we notice that for sufficiently small $\eps$,
\[
\Pp(A_3(\eps,T))\le\Pp\left\{(1\vee e^{\lambda_kT})T C^*>\frac{\eps^{\bar\alpha-2}}{4}\right\}=0,
\]
where $C^*=\sup|C(y)|$, which completes the proof.
\end{proof}

\begin{proof}[Sketch of a proof of Lemma~\ref{lm:tau=tau_1}.]  The proof is similar to the proof of Lemma~\ref{lm:tau-to-infty} 
and uses the fact that among the factors $e^{\lambda_k t}$, the one with $k=1$ 
grows fastest of all in~$t$, and $t$ is large due to  Lemma~\ref{lm:tau-to-infty},
so that the exit level $\eps^{\bar\alpha}$ is first reached by $|Y^1_\eps(t)-(S^ty_0)^1|$ with high
probability.
\end{proof}

\section{Proof of properties of the space of curves $\X$}\label{sec:basics-on-paths}
Any nondecreasing continuous map of $[0,1]$ onto itself is called a time change.
For a time change $\lambda$ we define its inverse by
\[
\lambda^{-1}(s)=\inf\{t:\ \lambda(s)\ge s\}.
\]
The inverse function is nondecreasing and satisfies 
\[\lambda\circ\lambda^{-1}(s)=s,\quad s\in[0,1].\] 
It is continuous on the left and has limits on the right, but
it may have jumps associated
with segments of  constancy, and therefore it is not necessarily a time change. Nevertheless,
 the following lemma holds true.
\begin{lemma}
\label{lm:nested_segments}
 Let $\lambda_1$ and $\lambda_2$ be two time changes. Suppose that each segment
of constancy of $\lambda_1$ is contained in a segment of constancy of $\lambda_2$. Then
$\lambda_2\circ\lambda_1^{-1}$ is a time change.
\end{lemma}
\begin{proof} The function  $\lambda_2\circ\lambda_1^{-1}$ is obviously nondecreasing with
$\lambda_2\circ\lambda_1^{-1}(0)=0$. Let us prove that it is continuous.
Take any point $s\in[0,1]$ and let 
\[[t_-,t_+]=\{t:\lambda_1(t)=s\}.\]
Since $\lambda_1^{-1}(s+)=t_+$ and $\lambda_1^{-1}(s-)=t_-$, we have
\[
\lambda_2\circ\lambda_1^{-1}(s+)=\lambda_2(t_+)
\]
and
\[
\lambda_2\circ\lambda_1^{-1}(s-)=\lambda_2(t_-).
\]
By the assumption of the lemma, $\lambda_2(t_+)=\lambda_2(t_-)$. Therefore,
\[
\lambda_2\circ\lambda_1^{-1}(s+)=\lambda_2\circ\lambda_1^{-1}(s-),
\]
and the continuity is proven.

If $s=1$, then, using the same notation and reasoning we see that $t_+=1$, and
\[
\lambda_2\circ\lambda_1^{-1}(1-)=\lambda_2(t_-)=\lambda_2(1)=1,
\]
and the proof is complete.
\end{proof}

\begin{lemma}\label{lm:factor_of_no_constancy} For any nonconstant path $\gamma$ there is a path $\gamma'$ with no intervals of
constancy, and a time change $\lambda$ so that $\gamma=\gamma'\circ\lambda$. 
\end{lemma}
\begin{proof} Let $I_1,I_2,\ldots$ be maximal nondegenerate segments of constancy of~$\gamma$. By assumption, none of these segments coincides with $[0,1]$.
Therefore, there is a time change $\lambda$
that has these and only these maximal nondegenerate segments of constancy ($\lambda$ may be constructed analogously to the Cantor staircase). 
It is easy to see that $\gamma'=\gamma\circ\lambda^{-1}$ is continuous.

Let us assume that there is a nodegenerate segment $[s_-,s_+]$ and a point $x$, such that $\gamma'(s)=x$ for all $s\in[s_-,s_+]$.
This means that $\gamma(t)=x$ for all $t\in\lambda^{-1}([s_-,s_+])$. Now notice that $\lambda^{-1}([s_-,s_+])$ consists of the left ends
of all maximal (not necessarily nondegenerate) segments of constancy of $\lambda$ (and thus of $\gamma$) contained
in $[t_-,t_+]=[\lambda^{-1}(s_-),\lambda^{-1}(s_+)]$. So, for the left end $t$ of any interval of constancy contained in $[t_-,t_+]$, 
we have $\gamma(t)=x$. Therefore, in fact, $\gamma(t)=x$ for all $t\in[t_-,t_+]$. We conclude that $[t_-,t_+]$ is contained in some segment
of constancy of $\lambda$, so that
\[
s_-=\lambda(t_-)=\lambda(t_+)=s_+,
\]
which contradicts the nondegeneracy of $[s_-,s_+]$. This finishes the proof that $\gamma'=\gamma\circ\lambda^{-1}$ is a 
curve with no intervals of constancy. To see that $\gamma=\gamma'\circ\lambda$, we notice that
for each $s$, 
\[
\lambda^{-1}\circ\lambda(s)=\inf\{s_0:\lambda(s_0)=\lambda(s)\}
\]
is the left-end of the segment of constancy of $\lambda$ (and $\gamma$), containing $s$. Therefore,
\[
\gamma'\circ\lambda(s)=\gamma\circ\lambda^{-1}\circ\lambda(s)=\gamma(s), \quad s\in[0,1],
\]
and the proof is complete.
\end{proof}

\begin{proof}[Proof of Lemma \ref{lm:equivalence}] Of all equivalence properties, only the transitivity is not quite obvious. So, we assume that
$\gamma_1\sim\gamma_2$ and $\gamma_2\sim\gamma_3$ and prove that $\gamma_1\sim\gamma_3$. In other words, our 
assumption is that there are paths $\gamma_{12}$, $\gamma_{23}$ and time changes $\lambda_{12,1}$, $\lambda_{12,2}$
$\lambda_{23,2}$, $\lambda_{23,3}$ such that
\begin{align*}
\gamma_1&=\gamma_{12}\circ\lambda_{12,1},&\quad \gamma_2&=\gamma_{12}\circ\lambda_{12,2},\\
\gamma_2&=\gamma_{23}\circ\lambda_{23,2},&\quad \gamma_3&=\gamma_{23}\circ\lambda_{23,3}.
\end{align*}

According to Lemma~\ref{lm:factor_of_no_constancy}, there is a path $\gamma'$ with no intervals of constancy, and a time change $\lambda$
such that
\[
\gamma_2=\gamma'\circ\lambda=\gamma_{12}\circ\lambda_{12,2}=\gamma_{23}\circ\lambda_{23,2}.
\]
This representation implies that segments of constancy of time changes $\lambda_{12,2}$ and $\lambda_{23,2}$
are contained in segments of constancy of $\lambda$, so that Lemma~\ref{lm:nested_segments} implies that
$\lambda\circ\lambda_{12,2}^{-1}$ and $\lambda\circ\lambda_{23,2}^{-1}$ are time changes. We also see that
\begin{align*}
\gamma_{12}&=\gamma'\circ\lambda\circ\lambda_{12,2}^{-1},\\
\gamma_{23}&=\gamma'\circ\lambda\circ\lambda_{23,2}^{-1},
\end{align*}
so that
\begin{align*}
\gamma_{1}&=\gamma'\circ\lambda\circ\lambda_{12,2}^{-1}\circ\lambda_{12,1},\\
\gamma_{3}&=\gamma'\circ\lambda\circ\lambda_{23,2}^{-1}\circ\lambda_{23,2},
\end{align*}
which completes the proof, since the paths $\gamma_1,\gamma_3$ are represented via a common underlying path $\gamma'$ and time changes
$\lambda\circ\lambda_{12,2}^{-1}\circ\lambda_{12,1}$ and $\lambda\circ\lambda_{23,2}^{-1}\circ\lambda_{23,2}$,
respectively.
\end{proof}
\bigskip


For a curve $\Gamma\in\X$, we shall denote by $\Gamma'$ the set of all $\gamma\in\Gamma$ with no segments
of constancy. Lemma~\ref{lm:factor_of_no_constancy} shows that $\Gamma'\ne\emptyset$ for any curve
represented by a nonconstant path.

For a continuous function $f:[0,1]\to \R^{d+1}$, we use 
\[
|f|_\infty=\sup_{s\in[0,1]}|f(s)|.
\] 

\begin{lemma}\label{lm:basics-on-paths}
\begin{enumerate}
\item\label{it:existence-of-reparametrization}  For every $\gamma_1\in\Gamma$ and $\gamma_2\in\Gamma'$, there is a time change
$\lambda$ such that $\gamma_1=\gamma_2\circ \lambda$. If $\gamma_1\in\Gamma'$,
then $\lambda$ is a bijection.
\item\label{it:approx-by-bijections} For every $\gamma\in\Gamma$ and every $\eps>0$, there is a $\gamma_\eps\in\Gamma'$ 
such that \[|\gamma-\gamma_\eps|_\infty<\eps.\]
\item\label{it:path_is_a_closed_set} Every curve is a closed set in the sup-norm $|\cdot|_\infty$.
\item\label{it:equivalent_def} The definition \eqref{eq:def-distance} is equivalent to
\begin{equation}
\label{eq:def-distance2}
d(\Gamma_1,\Gamma_2)=\inf_{\gamma_1\in\Gamma'_1, \gamma_2\in\Gamma'_2} |\gamma_1-\gamma_2|_\infty.
\end{equation}
\end{enumerate}
\end{lemma}
\begin{proof}
Part~\ref{it:existence-of-reparametrization}. If $\gamma_1\in\Gamma$, $\gamma_2\in\Gamma'$, then there are time changes
 $\lambda_1,\lambda_2$ and a path $\gamma'$ 
with $\gamma_1=\gamma'\circ \lambda_1$ and
$\gamma_2=\gamma'\circ\lambda_2$. 
The map $\lambda_2$ is strictly increasing and $\gamma'\in\Gamma'$, since if any of these
two conditions is violated then $\gamma_2$ has a segment of constancy. In particular, $\lambda_2^{-1}$ is also a time change, and
we can write $\gamma'=\gamma_2\circ \lambda_2^{-1}$.
So, $\gamma_1=\gamma_2\circ\lambda_2^{-1}\circ\lambda_1$, and we can set
$\lambda=\lambda_2^{-1}\circ\lambda_1$. If $\lambda$ is not a bijection, then $\gamma_1$ has an
interval of constancy, and the proof of part~\ref{it:existence-of-reparametrization} is complete.

To prove part~\ref{it:approx-by-bijections}, we use part~\ref{it:existence-of-reparametrization} to find
$\gamma'\in\Gamma'$ and a time change $\lambda$ so that $\gamma=\gamma'\circ\lambda$. For any $\delta>0$,
$\lambda_\delta(s)=\delta s +(1-\delta)\lambda(s)$ defines a time change. Notice
that $|\lambda_\delta-\lambda|_\infty\le 2\delta$ and $\gamma' \circ \lambda_\delta\in\Gamma'$.
Due to the uniform continuity of $\gamma$,
\[|\gamma-\gamma' \circ \lambda_\delta|_\infty= |\gamma'\circ \lambda-\gamma' \circ \lambda_\delta|_\infty<\eps,\]
 for
sufficiently small $\delta$, and we are done.  

Part~\ref{it:path_is_a_closed_set}. We have to prove that if 
$\gamma_n\in\Gamma$ for all $n\in\N$, and $\lim_{n\to\infty}|\gamma_n-\mu|_\infty=0$, then $\mu\in\Gamma$.
Due to part~\ref{it:approx-by-bijections}, it is sufficient to consider the case where $\gamma_n\in\Gamma'$ for all $n$. In this situation
there exists
a sequence of time changes $\lambda_n$ and $\gamma\in\Gamma'$ such that $\gamma_n=\gamma\circ\lambda_n$
for all $n$.
Due to Helley's Selection Theorem, see \cite[Appendix II]{Billingsley:MR1700749}, there is a sequence $n'\to\infty$ and  
a nondecreasing function 
$\lambda_\infty:[0,1]\to[0,1]$ such that $\lim_{n'\to\infty}\lambda_{n'}(s)=\lambda_\infty(s)$ for every
point of continuity $s$ of $\lambda_\infty$. 
This implies 
\begin{equation}
\label{eq:limit_representation_of_mu}
\mu(s)=\lim_{n'\to\infty}\gamma_{n'}(s)=\lim_{n'\to\infty}\gamma\circ\lambda_{n'}(s)=\gamma\circ\lambda_\infty(s)
\end{equation}
for all points of continuity $s$.

Our goal is to show that $\lambda_\infty$ is actually
continuous at every point in $[0,1]$. That will imply that $\lambda_{n'}$ converges to $\lambda_\infty$ uniformly, since all these functions
are nondecreasing, and we shall be able to conclude that $\lambda_\infty(0)=0$, $\lambda_\infty(1)=1$, so that
$\lambda_\infty$ is a time change. Moreover, it follows that~\eqref{eq:limit_representation_of_mu} holds for all $s\in[0,1]$, so that
$\mu=\gamma\circ\lambda_\infty$ and $\mu\in\Gamma$.

So, we assume that there is a point $s_0\in(0,1)$ with $\lambda_\infty(s_0-)<\lambda_\infty(s_0+)$ and we are
going to show that this assumption contradicts the uniform convergence of $\gamma_{n'}$ to $\mu$. 

Take any point $r\in(\lambda_\infty(s_0-),\lambda_\infty(s_0+))$ with 
$\gamma(r)\ne\gamma\circ\lambda_\infty(s_0+)$ and $\gamma(r)\ne\gamma\circ\lambda_\infty(s_0-)$ (this can be done since
$\gamma$ has no segments of constancy),
and define
\begin{equation}
\label{eq:def-of-eps-for-non-uniform}
\eps=\frac{1}{2}\min\{|\gamma(r)-\gamma\circ\lambda_\infty(s_0+)|,|\gamma(r)-\gamma\circ\lambda_\infty(s_0-)|\}.
\end{equation} 

The continuity of $\gamma$ and the existence of the right and left limits of $\lambda_\infty$ at $s_0$ allow
us to choose $\delta>0$ so that 
\begin{align}
s\in(s_0,s_0+\delta)\quad \text{implies\quad } 
|\gamma\circ\lambda_\infty(s)-\gamma\circ\lambda_\infty(s_0+)|<\eps,
\label{eq:close-to-the-jump1}\\
s\in(s_0-\delta,s_0)\quad \text{implies\quad }
|\gamma\circ\lambda_\infty(s)-\gamma\circ\lambda_\infty(s_0-)|<\eps.
\label{eq:close-to-the-jump2}
\end{align}

Since $\lambda_{n'}$ converges to $\lambda_\infty$ at all its continuity points, and these are dense, we see that 
there is a number $n_0$ such that for every $n'>n_0$, there is a continuity point $s(n')\in(s_0-\delta,s_0+\delta)$ with $\lambda_{n'}(s(n'))=r$.
For these $n'$, if $s=s(n')\in(s_0,s_0+\delta)$, \eqref{eq:def-of-eps-for-non-uniform} and \eqref{eq:close-to-the-jump1} imply
\begin{align*}
|\gamma_{n'}(s)-\gamma\circ\lambda_\infty(s)|&=|\gamma(r)-\gamma\circ\lambda_\infty(s)|\\
&\ge |\gamma(r)-\gamma\circ\lambda_\infty(s_0+)|-|\gamma\circ\lambda_\infty(s)-\gamma\circ\lambda_\infty(s_0+)|\\
&\ge 2\eps-\eps=\eps,
\end{align*}
and if $s=s(n')\in(s_0-\delta,s_0)$, \eqref{eq:def-of-eps-for-non-uniform} and \eqref{eq:close-to-the-jump2} imply
\begin{align*}
|\gamma_{n'}(s)-\gamma\circ\lambda_\infty(s)|&=|\gamma(r)-\gamma\circ\lambda_\infty(s)|\\
&\ge |\gamma(r)-\gamma\circ\lambda_\infty(s_0-)|-|\gamma\circ\lambda_\infty(s)-\gamma\circ\lambda_\infty(s_0-)|\\
&\ge 2\eps-\eps=\eps,
\end{align*}
So, for all $n'>n_0$ there is $s\in(s_0-\delta,s_0+\delta)$ such that
\[
|\gamma_{n'}(s)-\mu(s)|=|\gamma_{n'}(s)-\gamma\circ\lambda_\infty(s)|\ge\eps,
\]
and we obtain a contradiction with the uniform convergence of $\gamma_{n'}$ to $\mu$. This completes the proof that
there is no discontinuities of $\lambda_\infty$ within $(0,1)$. The demonstration showing that $\lambda_\infty$ is
continuous at the endpoints of $[0,1]$ is similar. This completes the proof of part~\ref{it:path_is_a_closed_set}.

Part~\ref{it:equivalent_def} follows immediately from part \ref{it:approx-by-bijections}. The lemma is proven completely.
\end{proof}


\begin{proof}[Proof of Theorem~\ref{thm:metric_space}] Part~1. We notice first that $d$ is nonnegative and symmetric.
To show the triangle inequality, for any three curves $\Gamma_1,\Gamma_2,\Gamma_3$ and $\eps>0$ we use 
part~\ref{it:equivalent_def} of Lemma~\ref{lm:basics-on-paths} to
find
$\gamma_{1,2}\in\Gamma'_1$, $\gamma_{2,1},\gamma_{2,3}\in\Gamma'_2$, and $\gamma_{3,2}\in\Gamma'_3$
with \[|\gamma_{1,2}-\gamma_{2,1}|_\infty<d(\Gamma_1,\Gamma_2)+\eps\] and
\[|\gamma_{3,2}-\gamma_{2,3}|_\infty<d(\Gamma_3,\Gamma_2)+\eps.\] Let us find an  time change $\lambda$ such that
$\gamma_{2,3}=\gamma_{2,1}\circ\lambda$. Then
\begin{align*}
d(\Gamma_1,\Gamma_3)&\le|\gamma_{1,2}\circ\lambda-\gamma_{3,2}|_\infty\\
&\le|\gamma_{1,2}\circ\lambda-\gamma_{2,1}\circ\lambda|_\infty+|\gamma_{2,1}\circ\lambda-\gamma_{3,2}|_\infty\\
&\le|\gamma_{1,2}-\gamma_{2,1}|_\infty+|\gamma_{2,3}-\gamma_{3,2}|_\infty\\
&\le  d(\Gamma_1,\Gamma_2)+d(\Gamma_3,\Gamma_2)+2\eps,
\end{align*}
and the desired inequality follows since $\eps$ is arbitrarily small.

If $d(\Gamma_1,\Gamma_2)=0$, then there are two sequences of paths $\gamma_{1,n}\in\Gamma'_1$,
$\gamma_{2,n}\in\Gamma'_2$ with $\lim_{n\to\infty}|\gamma_{1,n}-\gamma_{2,n}|_\infty=0$. Using appropriate time changes,
we see that there is $\tilde \gamma_1\in\Gamma'_1$ and $\tilde\gamma_{2,n}\in\Gamma'_2, n\in\N$ such that
$\lim_{n\to\infty}|\tilde\gamma_1-\tilde\gamma_{2,n}|=0$. Due to part \ref{it:path_is_a_closed_set} of Lemma~\ref{lm:basics-on-paths}, we see that $\tilde\gamma_1\in\Gamma_2$. Therefore,
$\Gamma_1=\Gamma_2$, and we are done with part~1.

Part 2. The space $\X$ is separable since it inherits a dense countable set from a closed
subset (of functions with nondecreasing zeroth coordinate) of the separable space $C([0,1]\to [0,\infty)\times\R^d)$. 

Let us now prove that $\X$ is complete.
Suppose that $(\Gamma_n)_{n\in\N}$ is a Cauchy sequence in $(\X,\rho)$. Using Lemma~\ref{lm:basics-on-paths} and 
reparametrization, we
can find an increasing number sequence $(n_k)_{k\in\N}$, with $n_k\in\N$, and 
a sequence of paths $(\gamma_k)_{k\in\N}$ with $\gamma_k\in\Gamma'_{n_k}$ such that
$|\gamma_k-\gamma_{k+1}|_\infty<2^{-k}$. This sequence of paths is Cauchy in $C$ 
and, therefore,  converges to a path $\gamma_\infty$ which defines a path $\Gamma_\infty$. Obviously,
$\Gamma_{n_k}$ converges to $\Gamma_\infty$ as $k\to\infty$. One can now use
the triangle inequality to extend this convergence to the whole sequence~$(\Gamma_n)$. So every
Cauchy sequence is convergent, which completes the proof of Theorem~\ref{thm:metric_space}.
\end{proof}
%

\begin{proof}[Proof of Lemma~\ref{lm:graph_coverge_if_continuous_function_converge}] 
We prove the necessity part of the statement since the sufficiency part is obvious. Due to Lemma~\ref{lm:basics-on-paths}, we may assume that there is a sequence of paths
$\gamma_n\in\Gamma_{f_n}'$ such that 
\begin{equation}
\label{eq:uniform_convergence_of_paths}
|\gamma_n-\gamma_g|_\infty\to 0,\quad n\to\infty,
\end{equation} 
where
$\gamma_g$ is defined in~\eqref{eq:gamma_f} .
Therefore,
\[
\sup_{s\in[0,1]}|\gamma_n^0(s)-sT|\to 0,\quad n\to\infty.
\]
This in turn implies that if we define a strictly increasing and continuous function $u_n(s)$ by $\gamma_n^0(u(s))=sT$,
then
\[
\sup_{s\in[0,1]}|u_n(s)-sT|\to 0,\quad n\to\infty.
\]
This, together with~\eqref{eq:uniform_convergence_of_paths} and uniform continuity of $g$ proves that
\[
\sup_{s\in[0,1]}|f_n(sT)-g(sT)|=\sup_{s\in[0,1]}|\gamma_n(u_n(s))-\gamma_g(s)|\to 0,\quad n\to\infty,
\]
and the proof is complete.
\end{proof}

\begin{proof}[Proof of Lemma~\ref{lm:proximity_to_piecewise_constant}] It is straightforward to see that
the path $\tilde\gamma_f$ defined by
\[
\tilde\gamma_f(s)=\begin{cases}(\gamma^0(s),f^1(\gamma^0(s)),\ldots,f^d(\gamma^0(s))),&\gamma^0(s)\in[r_{2j-1},r_{2j}],\ j=1,\ldots,k\\
                             (r_{2j}, f^1(r_{2j}),\ldots,f^d(r_{2j})),&\gamma^0(s)\ge r_{2j},\ s\le s'_{2j},\ j=0,\ldots,k\\
                             (\lambda_j^{-1}(s),f^1(\lambda_j^{-1}(s)),\ldots,f^d(\lambda_j^{-1}(s))),& s\in[s'_{2j},s'_{2j+1}], \ j=0,\ldots,k\\
			     (r_{2j+1}, f^1(r_{2j+1}),\ldots,f^d(r_{2j+1})),&s\ge s'_{2j+1},\ \gamma^0(s)\le r_{2j+1},\ j=0,\ldots,k
			     \end{cases}
\]
belongs to $\Gamma_f$ and
$|\gamma-\tilde\gamma_f|_\infty<3\delta$.
\end{proof}
 
\bibliographystyle{plain}
\bibliography{happydle}
\end{document}